\documentclass[10pt]{amsart}
\usepackage{amsmath, amssymb, amsthm, bbold}
\usepackage{paralist, xcolor, tikz, hyperref}
\usepackage{young}
\usepackage[margin=1in,letterpaper,portrait]{geometry}
\usepackage{enumitem}
\usepackage{caption}
\usepackage[noadjust]{cite}
\usepackage{comment}
\usepackage{tikz, tikz-cd} 
 \usetikzlibrary{braids,arrows.meta}
\usepackage[normalem]{ulem}
\usepackage{xurl,hyperref}

\theoremstyle{plain}
\newtheorem{conjecture}{Conjecture}[section]
\newtheorem{theorem}[conjecture]{Theorem}
\newtheorem{corollary}[conjecture]{Corollary}
\newtheorem{lemma}[conjecture]{Lemma}
\newtheorem{proposition}[conjecture]{Proposition}

\newtheorem{question}[conjecture]{Question}
\newtheorem{example}[conjecture]{Example}

\newtheorem{definition}[conjecture]{Definition}
\newtheorem*{question*}{Question}

\newtheorem{innercustomthm}{Theorem}
\newenvironment{customthm}[1]
  {\renewcommand\theinnercustomthm{#1}\innercustomthm}
  {\endinnercustomthm}

\newtheorem{innercustomcor}{Corollary}
\newenvironment{customcor}[1]
  {\renewcommand\theinnercustomcor{#1}\innercustomcor}
  {\endinnercustomcor}

\newcommand{\prefix}{\operatorname{Pre}}

\newcommand{\voting}{\rho}
\newcommand{\majrule}[1]{\prec^{\mathrm{maj}}_{#1}}

\newcommand{\Inv}{\operatorname{Inv}}

\newcommand{\supp}{\operatorname{supp}}

\newcommand{\ceven}{c_{\mathrm{even}}}
\newcommand{\codd}{c_{\mathrm{odd}}}

\newcommand{\ii}{\mathbf{i}}
\newcommand{\jj}{\mathbf{j}}

\newcommand{\lf}{\mathbf{m}}

\newcommand{\one}{\mathbb{1}}

\newcommand{\DDD}{\mathcal{D}}

\newcommand{\NN}{\mathbb{N}}

\newcommand{\RR}{\mathbb{R}}

\author{Victor Reiner}
\address{School of Mathematics, University of Minnesota, Minneapolis MN, USA}
\email{reiner@umn.edu}

\author{Bridget Eileen Tenner}   
\address{Department of Mathematical Sciences, DePaul University, Chicago, IL, USA}
\email{bridget@math.depaul.edu}


\title{Majority relations for Condorcet domains of tiling type}

\begin{document}

\begin{abstract}
Condorcet domains are subsets of permutations arising in voting theory: regarding their
permutations as preference orders on a list of candidates,
one avoids Condorcet's paradox when
aggregating the preferences
via a simple {\it majority relation}. We use poset theory to show that, for the subclass of Condorcet domains of tiling type, the majority rule has stronger properties.  We then develop techniques to predict the majority rule
explicitly for the uniform vote tally on
Condorcet domains of tiling type, and
apply this to several well-known
examples.   
\end{abstract}

\maketitle

\section{Introduction, background, and examples}
\label{sec:intro}

Voting theory sometimes models a voter's preference among a set of candidates $[n]:=\{1,2,\ldots,n\}$
as a permutation $w=(w_1,\ldots, w_n)$ in the symmetric group $S_n$.  One thinks of $w$ as specifying the preference order
$w_1 <_w w_2 <_w \cdots <_w w_n$ among the candidates, so that $<_w$ is a complete, antisymmetric, transitive binary relation on $[n]$.
Given any function $\voting: S_n \rightarrow \NN:=\{0,1,2,\ldots\}$ that specifies the {\it vote tally} or a {\it voting profile}, where $\voting(w)$ is the number of voters with preference order $w$, there are various schemes to aggregate these preferences into a single binary relation on $[n]$.  
This paper focuses on one such scheme, 
called the {\it (strict) majority} relation (see, e.g., Monjardet \cite{Monjardet}, Puppe and Slinko \cite{PuppeSlinko}).  It is the binary relation $\majrule{\voting}$ on $[n]$ defined as follows: 
\begin{equation*}
a \majrule{\voting} b\quad  \text{ if and only if }
\sum_{\substack{w \in S_n:\\ a <_w b}} \voting(w) >  \sum_{\substack{w \in S_n:\\ b <_w a}} \voting(w),
\text{ equivalently, if and only if }
\sum_{\substack{w \in S_n:\\ a <_w b}} \voting(w) >  \frac{1}{2} |\voting|
\end{equation*}
where $|\voting|:=\sum_{w \in S_n} \voting(w)$  is the total vote tally;
that is, the number of voters.

The binary relation $\majrule{\voting}$ is always {\it antisymmetric}, meaning that it has no $2$-cycles $a \majrule{\voting} b \majrule{\voting} a$.  However, it is susceptible to {\it Condorcet's paradox}, allowing longer cycles\footnote{Warning: In \cite{Monjardet, PuppeSlinko} the authors use more neutral symbols like $R^{\mathrm{maj}}_\voting$ rather than $\majrule{\voting}$, due to the possibility of these cycles.}.
For example, if $\voting(w)=1$ for each of the three linear orders $\{(1,2,3),(2,3,1),(3,1,2)\}$ in $S_3$, and $\voting(w)=0$ for all other $w$ in $S_3$, then
$1 \majrule{\voting} 2 \majrule{\voting} 3 \majrule{\voting} 1$.
This problem is avoided when one restricts
the {\it support set} of $\voting$
\begin{equation*}
\supp(\voting):=\{w \in S_n: \voting(w) > 0\}
\end{equation*}
to lie inside what is known as a {\it Condorcet} or {\it acyclic domain} $\DDD$;
that is, whenever, for each triple $\{a,b,c\} \subset [n]$, there do not exist $u,v,w$ in $\DDD$ with 
$a <_u b <_u  c$ and $b  <_v c <_v  a$ and $c <_w a <_w  b$.
See Monjardet \cite{Monjardet}, Puppe and Slinko \cite{PuppeSlinko} for useful surveys and background on Condorcet domains and majority rule relations.

There is a substantial literature on constructing
{\it large} Condorcet domains $\DDD$ inside $S_n$. 
One such construction includes work
of Chameni-Nembua \cite{Chameni-Nembua}, Abello \cite{Abello}, Galambos and Reiner \cite{GalambosReiner}, and Danilov, Karzanov, and Koshevoy \cite{DanilovKarzanovKoshevoy, DanilovKarzanovKoshevoy2021},
These last authors referred to the domains $\DDD$ produced via this construction as {\it Condorcet domains of tiling type}.  A special case includes the {\it alternating schemes} introduced by
Fishburn \cite{Fishburn1997, Fishburn2002}.

The goal of this paper is to point out that
the majority relation $\majrule{\voting}$
has much better behavior and is easier to compute when one restricts the support set $\supp(\voting)$ to equal {\it on the nose} a Condorcet domain of tiling type.  In this context, the computation of $\majrule{\voting}$ will be reformulated in Section~\ref{sec:heaps} below using a connection between Condorcet domains $\DDD$ of tiling type and the theory of {\it posets}.
We will also be particularly
interested in computing $\majrule{\voting}$ explicitly for the special case where $\voting=\one_\DDD$ is the
{\it uniform vote tally} defined by $\one_\DDD(w)=1$ if $w$ lies in $\DDD$, and $\one_\DDD(w)=0$ for $w$ in $S_n \setminus \DDD$.
In this uniform case, if the Condorcet domain $\DDD$ of tiling type has a certain symmetry,
Section~\ref{sec:horizontal-folds} will show how
to calculate $\majrule{\one_\DDD}$ explicitly.  Section~\ref{sec:examples} then applies this work to several important examples, including Fishburn's alternating scheme.  In another direction, Section~\ref{sec:disconnected-heaps} reduces the computation to the case of {\it indecomposable} permutations in $S_n$, that is, those that belong to no proper Young subgroup $S_{n_1} \times S_{n_2}$. Finally, Section~\ref{sec:remarks-and-questions} presents a direction for future research about the relationship between majority relations and the higher Bruhat order $B(n,2)$.

\medskip
In the remainder of this introduction, we define Condorcet domains of tiling type and explain their relation to posets more fully, illustrating some of our results with a running example.  

The story starts with the {\it Coxeter presentation} of the symmetric group, which uses the generating set $S=\{s_1,s_2,\ldots,s_{n-1}\}$ where $s_i$ is the adjacent transposition swapping $i \leftrightarrow i+1$,
satisfying these relations:
\begin{align}
s_i^2=1 &\text{ for }i=1,2,\ldots,n-1,\\
\label{eq:commuting-relation}
s_is_j=s_js_i &\text{ for }|i-j| \geq 2,\\
\label{eq:Yang-Baxter-relation} s_is_{i+1}s_i=s_{i+1}s_is_{i+1} &\text{ for }i=1,2,\ldots,n-2.
\end{align}
Permutations $w$ in $S_n$ can have
many factorizations $w=s_{i_1} s_{i_2} \cdots s_{i_\ell}$ with $s_i$ in $S$.  The
shortest factorizations
have length $\ell=:\ell(w)=\#\Inv(w)$, the number of {\it inversions} of $w$,
equivalently, the size of its {\it inversion set}
\begin{equation*}
\Inv(w):=\{(a,b): 1 \leq a<b \leq n \text{ and }b <_w a\}.
\end{equation*}
One calls the sequence $\ii=(i_1,i_2,\ldots,i_\ell)$ of subscripts appearing in one of these shortest factorizations $w=s_{i_1} s_{i_2} \cdots s_{i_\ell}$
a {\it reduced word} for $w$.  Given any reduced word $\ii$ for $w$,
a theorem of Matsumoto \cite{Matsummoto} and Tits \cite{Tits} (see Bj\"orner and Brenti \cite[Thm.~3.3.1]{BjornerBrenti}) shows how to produce
all of the others:  any two reduced words $\ii$ and $\ii'$ for $w$ can be connected by a sequence $\ii=\ii_0 - \ii_1 - \cdots -\ii_{m-1} - \ii_m=\ii'$ in which each step 
$\ii_j - \ii_{j+1}$ applies either a {\it commuting move} as in~\eqref{eq:commuting-relation} or a {\it Yang-Baxter relation} as in~\eqref{eq:Yang-Baxter-relation}.  When one only allows the commuting moves
from~\eqref{eq:commuting-relation}, one calls the reduced words obtainable from $\ii$ by a sequence of such moves its {\it commuting equivalence class} or \emph{commutation class}, denoted $C(\ii)$. 

\begin{definition} \rm
\label{def:Condorcet-domain-of-tiling-type}
    A {\it Condorcet domain $\DDD$ of tiling type} corresponds to the choice of a permutation\footnote{We are actually allowing a slightly broader definition here than the one in Danilov, Karzanov, and Koshevoy \cite{DanilovKarzanovKoshevoy}.
    They additionally required that $w=w_0=(n,n-1,\ldots,3,2,1)$, the longest permutation, so that the resulting domain $\DDD=\prefix(C)$ is {\it complete/inclusion-maximal}, meaning that $\DDD \sqcup \{v\}$ is never a Condorcet domain for any $v \in S_n \setminus \DDD$.} $w$ in $S_n$ along with the choice of a  commuting equivalence class $C=C(\ii)$ for one of its reduced words $\ii=(i_1,\ldots,i_\ell)$.
 The domain $\DDD$ is defined to be the set $\prefix(C)$ of all permutations $w'=s_{i_1'} s_{i'_2} \cdots s_{i'_m}$, for which $(i_1',i'_2,\ldots,i'_m)$ is a prefix 
of a reduced word $\ii'=(i'_1,i_2',\ldots,i'_{\ell})$ in the same 
commuting equivalence class $C$. In other words,
$$\DDD=\prefix(C) := \{s_{i_1'}s_{i_2'} \cdots s_{i'_{m}} :
(i_1',i'_2,\ldots,i'_m) \text{ is the prefix of an element } \ii' \text{ in }C\}.$$
\end{definition}

Such domains $\DDD=\prefix(C)$ have a more concrete description as a consequence of
Cartier and Foata's theory of
{\it partial commutation monoids} \cite{CartierFoata}, later called {\it heaps of pieces} by Viennot \cite{Viennot}; see also Stanley \cite[Exer.~3.123]{Stanley-EC1}.
One forms the {\it heap poset} $P_C$ from any reduced word $\ii=(i_1,\ldots,i_\ell)$ in $C$, whose
elements are the positions $\{1,2,\ldots,\ell\}$, and where $j <_{P_C} k$ if $j < k$ in $\NN$ and $s_{i_j}$ and $s_{i_k}$ do not commute, that is, if $|i_j-i_k| \leq 1$.  It is helpful to label
the element $j$ in the Hasse diagram of $P_C$ by $s_{i_j}$, which allows one to read off  $\prefix(C)$
from the distributive lattice $J(P_C)$ of all {\it order ideals} $J \subseteq P_C$:  for each order ideal $J$, reading its labels $s_{i_j}$ in the order of any linear
extension gives a permutation $w'=s_{i'_1}s_{i'_2} \cdots s_{i'_{\ell'}}$ lying in $\DDD$.
It is also helpful to consider a third labeling of $P_C$, labeling the element $j$ by the unique inversion pair $ab$ in $\Inv(w)$ for which
$$
\{(a,b)\}=
\Inv(s_{i_1} s_{i_2} \cdots s_{i_{j-1}} s_{i_j}) \setminus \Inv(s_{i_1} s_{i_2} \cdots s_{i_{j-1}} ). 
$$
This last labeling will be so useful that we will often regard $P_C$ as a poset on $\Inv(w)$.

We will thus have four different ways of referring to the specific object of our interest: it is a domain $\DDD \subset S_n$, which is precisely a Condorcet domain of tiling type $\prefix(C)$, and which is equal to the support set $\supp(\voting)$ of our vote tally $\voting$. Moreover, the distributive lattice $J(P_C)$ of order ideals in the heap poset $P_C$ can be labeled by the elements of $\prefix(C)$, and so we will also sometimes find it convenient to write $J(P_C)$ for $\DDD = \prefix(C) = \supp(\voting)$. 

\begin{example} \rm
\label{ex:intro-example}
Consider $w=(3,4,1,2,7,5,6)$
in $S_7$, abbreviated $w=3412756$ when suppressing commas and parentheses causes no confusion. Then $$
\Inv(w)=\{(1,3),(1,4),(2,3),(2,4),(5,7),(6,7)\}=\{13,14,23,24,57,67\},
$$
again suppressing commas and parentheses, so that
 $\ell(w)=\#\Inv(w)=6$. One of its shortest factorizations $w=s_2 s_1 s_3 s_2 s_6 s_5$ has corresponding reduced word
$\ii=(2,1,3,2,6,5)$.  In fact, it is not hard to show that there are exactly $30 = 2\binom{6}{2}$
reduced words for this $w$, all obtained by shuffling either of the two sequences $(2,1,3,2),(2,3,1,2)$ with the sequence $(6,5)$, all lying in the same commuting equivalence class $C=C(\ii)$.  
For example, $\ii'=(6,2,3,5,1,2)$
is another reduced word for $w$,
and $C(\ii')=C=C(\ii)$.

One can write down $P_C$ along with its two other labelings, using $\ii$ or $\ii'$ or any other element
of $C$.  We choose to do it here using $\ii$, which we interpret as the following sequence of permutations in $S_7$ 
from the identity $1234567$ to $w=3412756$, visiting the permutations $w'$ factored by a prefix of $\ii$, and creating one more inversion (identified below the arrow) in their inversion sets $I(w')$ at each stage:
$$
1234567 \ 
\begin{matrix} s_2\\ \longmapsto \\ 23 \end{matrix}
\ 
1324567 \ 
\begin{matrix} s_1\\ \longmapsto \\ 13 \end{matrix}
\ 
3124567 \ 
\begin{matrix} s_3\\ \longmapsto \\ 24 \end{matrix}
\ 
3142567 \ 
\begin{matrix} s_2\\ \longmapsto \\ 14 \end{matrix}
\ 
3412567 \ 
\begin{matrix} s_6\\ \longmapsto \\ 67 \end{matrix}
\ 
3412576 \ 
\begin{matrix} s_5\\ \longmapsto \\ 57 \end{matrix}
\ 
3412756
$$
The corresponding three labelings of $P_C$ are as follows.
\begin{center}
\begin{minipage}{.3\textwidth}
 \hspace*{0.25\linewidth}
\begin{tikzpicture}
  [scale=.18,auto=left,every node/.style={circle,fill=black!20}]
  \node (1) at (0,0) {$1$};
  \node (2) at (-4,4) {$2$};
  \node (3) at (4,4) {$3$};
  \node (4) at (0,8) {$4$};
  \node (5) at (8,1) {$5$};
  \node (6) at (8,7) {$6$};
  \foreach \from/\to in {1/2,1/3,2/4,3/4,5/6}
    \draw (\from) -- (\to);
\end{tikzpicture}
\end{minipage}
\begin{minipage}{.3\textwidth}
 \hspace*{0.25\linewidth}
\begin{tikzpicture}
  [scale=.18,auto=left,every node/.style={circle,fill=black!20}]
  \node (si1) at (0,0) {$s_2$};
  \node (si2) at (-4,4) {$s_1$};
  \node (si3) at (4,4) {$s_3$};
  \node (si4) at (0,8) {$s_2$};
  \node (si5) at (8,1) {$s_6$};
  \node (si6) at (8,7) {$s_5$};
  \foreach \from/\to in {si1/si2,si1/si3,si2/si4,si3/si4,si6/si5}
    \draw (\from) -- (\to);
\end{tikzpicture}
\end{minipage}
\begin{minipage}{.3\textwidth}
 \hspace*{0.25\linewidth}
\begin{tikzpicture}
  [scale=.18,auto=left,every node/.style={circle,fill=black!20}]
  \node (23) at (0,0) {$23$};
  \node (13) at (-4,4) {$13$};
  \node (24) at (4,4) {$24$};
  \node (14) at (0,8) {$14$};
  \node (67) at (8,1) {$67$};
  \node (57) at (8,7) {$57$};
  \foreach \from/\to in {23/13,23/24,13/14,24/14,67/57}
    \draw (\from) -- (\to);
\end{tikzpicture}
\end{minipage}
\end{center}
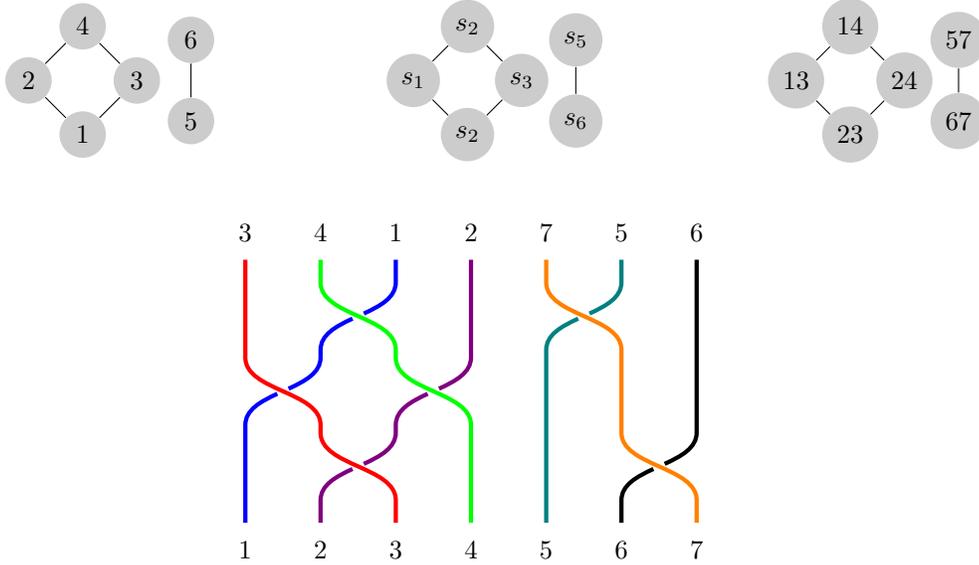
\begin{figure}[htbp]
\begin{tikzpicture}
\pic[
rotate=0,
braid/.cd,
every strand/.style={ultra thick},
strand 1/.style={red},
strand 2/.style={green},
strand 3/.style={blue},
strand 4/.style={violet},
strand 5/.style={orange},
strand 6/.style={teal}
] (coords) {braid={s_2-s_5 s_1-s_3 s_2-s_6}};
\node[ at=(coords-rev-1-e),  label=south : $1$ ] {} ;
\node[ at=(coords-rev-2-e),  label=south : $2$ ] {} ;
\node[ at=(coords-rev-3-e),  label=south : $3$ ] {} ;
\node[ at=(coords-rev-4-e),  label=south : $4$ ] {} ;
\node[ at=(coords-rev-5-e),  label=south : $5$ ] {} ;
\node[ at=(coords-rev-6-e),  label=south : $6$ ] {} ;
\node[ at=(coords-rev-7-e),  label=south : $7$ ] {} ;
\node[ at=(coords-rev-1-s),  label=north : $1$ ] {} ;
\node[ at=(coords-rev-2-s),  label=north : $2$ ] {} ;
\node[ at=(coords-rev-3-s),  label=north : $3$ ] {} ;
\node[ at=(coords-rev-4-s),  label=north : $4$ ] {} ;
\node[ at=(coords-rev-5-s),  label=north : $5$ ] {} ;
\node[ at=(coords-rev-6-s),  label=north : $6$ ] {} ;
\node[ at=(coords-rev-7-s),  label=north : $7$ ] {} ;
\end{tikzpicture}
\caption{The wiring diagram for the reduced word $\ii=(2,1,3,2,6,5)$ of the permutation $w=3412756$ in $S_7$.}
\label{fig:wiring disagram of 3412756 example}
\end{figure}
One can also view $P_C$ as a poset on the set of
{\it strand crossings} appearing in the {\it wiring diagram} associated to $\ii$;  see Figure~\ref{fig:wiring disagram of 3412756 example}.
Labeling $P_C$ via the inversions $ab$ in $\Inv(w)$ is the same as labeling the crossing of strand $a$ and strand $b$ by $ab$ in the wiring diagram shown in Figure~\ref{fig:wiring disagram of 3412756 example},
with the understanding $ab <_{P_c} cd$ if 
$\{a,b\} \cap \{c,d\}=\{x\}$ and the $ab$ inversion is created in $\ii$ before the $cd$ inversion.  In other words,
$ab <_{P_C} cd$ exactly when the crossings $ab$ and $cd$ occur on a common strand $x$, with $ab$ occurring earlier (lower on the page) than $cd$.

Figure~\ref{fig:distributive lattice of 3412756} depicts the distributive lattice $J(P_C)$ of order ideals $J \subseteq P_C$. Each element is labeled by the corresponding $w'$ in $\prefix(C)$
with a shortest factorization $s_{i'_1} \cdots s_{i'_{\ell'}}$
for $w'$ written above it. 
\begingroup
\Large
\begin{figure}[htbp]
\begin{tikzpicture}
  [scale=.15,auto=left]
  \node (e) at (0,0) {$\substack{\varnothing\\[.5ex]1234567}$};
  \node (2) at (-10,10) {$\substack{s_2\\[.5ex]1324567}$};
  \node (6) at (10,10) {$\substack{s_6\\[.5ex]1234576}$};
  \node (21) at (-20,20) {$\substack{s_2s_1\\[.5ex]3124567}$};
  \node (23) at (-10,20) {$\substack{s_2s_3\\[.5ex]1342567}$};
\node (26) at (0,20) {$\substack{s_2s_6\\[.5ex]{{\color{red}{u}=1324576}}}$};
  \node (65) at (20,20) {$\substack{s_6s_5\\[.5ex]1234756}$};
  \node (213) at (-20,30) {$\substack{s_2s_1s_3\\[.5ex]3142567}$};
  \node (216) at (-10,30) {$\substack{s_2s_1s_6\\[.5ex]3124576}$};
  \node (236) at (0,30) {$\substack{s_2s_3s_6\\[.5ex]1342576}$};
  \node (265) at (10,30) {$\substack{s_2s_6s_5\\[.5ex]1324756}$};
  \node (2132) at (-30,40) {$\substack{s_2s_1s_3s_2\\[.5ex]3412567}$};
  \node (2136) at (-10,40) {$\substack{s_2s_1s_3s_6\\[.5ex]{{\color{red}{v}=3142576}}}$};
  \node (2165) at (0,40) {$\substack{s_2s_1s_6s_5\\[.5ex]3124756}$};
  \node (2365) at (10,40) {$\substack{s_2s_3s_6s_5\\[.5ex]1342756}$};
    \node (21326) at (-20,50) {$\substack{s_2s_1s_3s_2s_6\\[.5ex]3412576}$};
  \node (21365) at (0,50) {$\substack{s_2s_1s_3s_6s_5\\[.5ex]3142756}$};
\node (213265) at (-10,60) {$\substack{s_2s_1s_3s_2s_6s_5\\[.5ex]{{w}}=3412756}$};
  \foreach \from/\to in {e/2,e/6,
  2/21,2/23,2/26,6/65,6/26,
21/216,21/213,23/213,23/236,
26/216,26/236,26/265,65/265,
213/2132,213/2136,216/2136,236/2136,236/2365,265/2365,216/2165,265/2165,
2132/21326,2136/21326,2136/21365,2165/21365,2365/21365,
21326/213265,21365/213265}
    \draw (\from) -- (\to);
\end{tikzpicture}
\caption{The distributive lattice of order ideals of $P_C$, where $C$ is the commutation class containing the reduced word $(2,1,3,2,6,5)$ of the permutation $3412756$. Each element is labeled by its corresponding permutation in $\prefix(C)$, together with a shortest factorization of that permutation.}
\label{fig:distributive lattice of 3412756}
\end{figure}
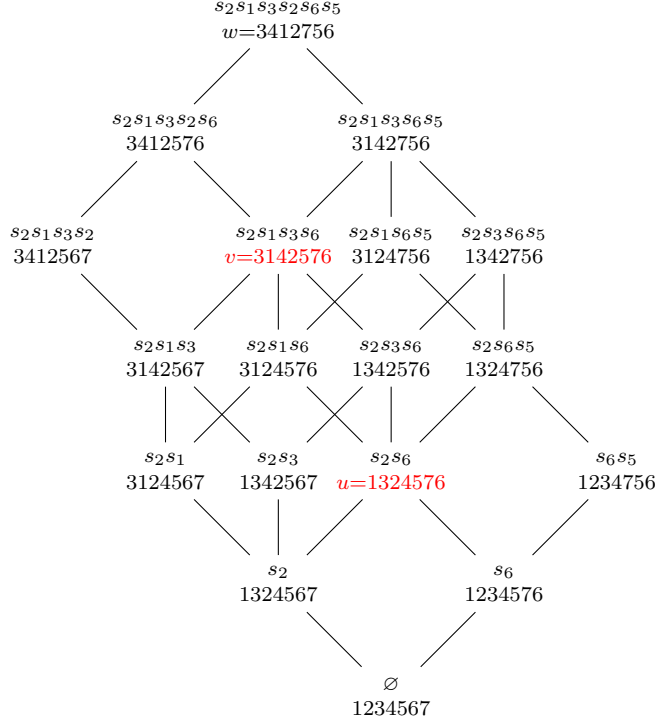
\endgroup
Two particular elements $u,v \in \prefix(C)$, which play a role in Examples~\ref{ex:intro-example-2}, \ref{ex:intro-example-3}, and~\ref{ex:intro-example-4} below, are identified in Figure~\ref{fig:distributive lattice of 3412756}.
\end{example}

It turns out that when $\supp(\voting)$ is a
Condorcet domain of tiling type, the majority relation $\majrule{\voting}$ is {\it nearly} a total order, in a precise
sense.

\begin{definition} \rm
Call a binary relation $\prec$ on $[n]$ a {\it prelinear order} if there exists an ordered set partition
$(B_1,B_2,\ldots,B_m)$ of $[n]=B_1 \sqcup B_2 \sqcup \cdots \sqcup B_m$
for which $a \prec b$ if and only if
$a \in B_r$ and $b \in B_s$ with $r<s$.
Say that this prelinear order $\prec$ has {\it only simple ties} if $|B_r|\leq 2$ for $r=1,2,\ldots,m$.
\end{definition}

Section~\ref{sec:heaps} will deduce the following corollary from the proof
of the main calculational result,
Theorem~\ref{thm:calcuating-maj-rule}.

\begin{corollary}
 \label{cor:no big/weird ties}   
When $\voting:S_n \rightarrow \NN$ has $\supp(\voting)$ that is a Condorcet domain of tiling type, then its majority relation $\majrule{\voting}$ is a prelinear order with only simple ties. 
\end{corollary}

The conclusion of
Corollary~\ref{cor:no big/weird ties} often fails for arbitrary Condorcet domains.  E.g., for each $n \geq 2$, consider the Condorcet domain $\mathcal{D}=\{e,w_0\} \subset S_n$ containing only the {\it identity permutation} $e:=(1,2,3,\ldots,n-1, n)$ 
and the {\it longest permutation} $w_0:=(n,n-1, \ldots,3,2,1)$.
Then the uniform vote tally $\voting=\one_\DDD$ has $\majrule{\voting}$ being the {\it empty} binary relation when considered as a 
subset of $[n] \times [n]$. That is, all
$n$ elements are ``tied" in $\majrule{\voting}$.

\begin{example} \rm
\label{ex:intro-example-2}
Consider the Condorcet domain $\prefix(C)$ of tiling type from Example~\ref{ex:intro-example}, with
uniform vote tally $\voting={\one_{\prefix(C)}}$.
With some brute force effort, one can check that its majority rule $\majrule{\voting}$ is
the prelinear order whose ordered set partition of $[7]$
has blocks
$
(B_1,B_2,B_3,B_4,B_5)
=(\{1,3\},\{2,4\},\{5\},\{7\},\{6\}).
$ 
That is, it consists of the permutations 
$$
1,3 \majrule{\voting}
2,4 \majrule{\voting}
5 \majrule{\voting}
7 \majrule{\voting}
6,
$$
with no relations between $1$ and $3$,
and no relations between $2$ and $4$.  We can abbreviate this prelinear order, suppressing commas, as 
$$\{1 \, 3\} \, \{2 \, 4\} \, 5 \, 7 \, 6.$$
Note that
this binary relation
$\majrule{\voting}$, when considered
as a subset of $[n] \times[n]$, 
is the intersection of the two linear orders $<_u$ and $<_v$, where
$u=1324576$ and $v=3142576$ in $\prefix(C)$ are labeled in Figure~\ref{fig:distributive lattice of 3412756}. This will be explained below.
\end{example}

Corollary~\ref{cor:no big/weird ties}
will follow from one our
main results, Theorem~\ref{thm:calcuating-maj-rule} below, showing that if $\supp(\voting)=\prefix(C)$ 
is of tiling type, then the computation of $\majrule{\voting}$ can be reduced to the following poset calculation.

\begin{definition} \rm
\label{def:voting-tally-function}
Given $\voting:S_n \rightarrow \NN$ having $\supp(\voting)=\prefix(C)$ of tiling type,
define the {\it $\voting$-tally function} 
$$
\begin{array}{rrcl}
\Sigma_\voting: & \Inv(w) &\longrightarrow &\{1,2,\ldots\}\\[1ex]
 & ab & \longmapsto &\Sigma_\voting(ab):=
\hspace{-.1in}\sum\limits_{\substack{w \in \prefix(C):\\ 
ab \in \Inv(w)}} \hspace{-.1in} \voting(w).
\end{array}
$$
\end{definition}

Recall that $|\voting|:=\sum_{w \in S_n} \voting(w)$ is the total number of voters.

\begin{theorem}
\label{thm:calcuating-maj-rule}
Assume $\voting: S_n \rightarrow \NN$
has $\supp(\voting)=\prefix(C)$ of tiling type. Regarding $P_C$ as a poset on $\Inv(w)$, 
the $\voting$-tally function $\Sigma_\voting: P_C \rightarrow \{1,2,\ldots\}$
is a strictly decreasing function,
that is, 
$$
ab <_{P_C} cd \text{ implies } \Sigma_\voting(ab) > \Sigma_\voting(cd).
$$
Furthermore, the majority relation
$\majrule{\voting}$, viewed as a subset of $[n] \times [n]$, is the intersection of two linear orders $<_u$ and $<_v$, with 
$u,v \in\prefix(C)$ defined uniquely by these inversion sets, which are order ideals in $P_C$:
\begin{align}
\Inv(u)&=\left\{ ab \in \Inv(w): \Sigma_\voting(ab)> \frac{1}{2}|\voting|\right\},\label{eqn:defining u as strict}\\
\Inv(v)&=\left\{ ab \in \Inv(w): \Sigma_\voting(ab) \geq \frac{1}{2}|\voting|\right\}.\label{eqn:defining v as weak}
\end{align}
\end{theorem}

\begin{example} \rm
\label{ex:intro-example-3}
Continuing Examples~\ref{ex:intro-example}
and~\ref{ex:intro-example-2}, let 
$\prefix(C)$ be the domain of tiling type defined there, and $\voting = \one_{\prefix(C)}$.  The $\voting$-tally function $\Sigma_\voting: P_C \rightarrow \{1,2,\ldots\}$
is depicted by the labeling of $P_C$ on the right, here:
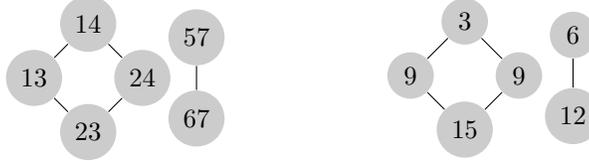
\begin{center}
\begin{minipage}{.3\textwidth}
 \hspace*{0.25\linewidth}
\begin{tikzpicture}
  [scale=.18,auto=left,every node/.style={circle,fill=black!20}]
  \node (23) at (0,0) {$23$};
  \node (13) at (-4,4) {$13$};
  \node (24) at (4,4) {$24$};
  \node (14) at (0,8) {$14$};
  \node (67) at (8,1) {$67$};
  \node (57) at (8,7) {$57$};
  \foreach \from/\to in {23/13,23/24,13/14,24/14,67/57}
    \draw (\from) -- (\to);
\end{tikzpicture}
\end{minipage}
\begin{minipage}{.3\textwidth}
\hspace*{0.25\linewidth}
\begin{tikzpicture}
  [scale=.18,auto=left,every node/.style={circle,fill=black!20}]
  \node (23) at (0,0) {$15$};
  \node (13) at (-4,4) {$9$};
  \node (24) at (4,4) {$9$};
  \node (14) at (0,8) {$3$};
  \node (67) at (8,1) {$12$};
  \node (57) at (8,7) {$6$};
  \foreach \from/\to in {23/13,23/24,13/14,24/14,67/57}
    \draw (\from) -- (\to);
\end{tikzpicture}
\end{minipage}
\end{center}
Note that  $|\voting|=|\DDD|=|J(P_C)|=18$, so $\frac{1}{2}|\voting|=9$.  Therefore the linear orders $<_u$ and $<_v$ are defined by
\begin{align*}
    \Inv(u)&=\{23,67\}=\{ab \in \Inv(w): \Sigma_\voting(ab) > 9\}, \text{ so that  } u=1324576, \text{ and}\\
    \Inv(v)&=\{23,13,24,67\} =\{ab \in \Inv(w): \Sigma_\voting(ab) \geq 9\}, \text{ so that  } v=3142576.
\end{align*}
Thus Theorem~\ref{thm:calcuating-maj-rule} computes $\majrule{\voting}$ as the intersection $<_u \cap <_v$,
agreeing with Example~\ref{ex:intro-example-2}:
$$
1,3 \majrule{\voting}
2,4 \majrule{\voting}
5 \majrule{\voting}
7 \majrule{\voting}
6.
$$
\end{example}

The remainder of the paper focuses on computing the majority rule $\majrule{\voting}$
for the uniform vote tally $\voting=\one_{\prefix(C)}$ on 
a Condorcet domain of tiling type $\prefix(C)$. 
Section~\ref{sec:horizontal-folds} shows
how to compute $\majrule{\voting}$ for $\voting=\one_{\prefix(C)}$ 
when the heap poset $P_C$ has what we call a {\it horizontal folding symmetry}.  Roughly speaking, this is an involutive poset anti-automorphism $\varphi: P_C \rightarrow P_C$, satisfying certain axioms, saying
$\varphi$ swaps related elements above and below a horizontal ``fold.''
This is illustrated below for $P_C$ as in Examples~\ref{ex:intro-example}, \ref{ex:intro-example-2}, and~\ref{ex:intro-example-3}.
\begin{center}
\begin{minipage}{.3\textwidth}
 \hspace*{0.25\linewidth}
\begin{tikzpicture}
  [scale=.18,auto=left,every node/.style={circle,fill=black!20}]
  \node (23) at (0,0) {$23$};
  \node (13) at (-4,4) {$13$};
  \node (24) at (4,4) {$24$};
  \node (14) at (0,8) {$14$};
  \node (67) at (8,1) {$67$};
  \node (57) at (8,7) {$57$};
\coordinate (leftfoldend) at (-8,4);
\coordinate (rightfoldend) at (12,4);
\draw[dashed] (leftfoldend) -- (rightfoldend);
  \foreach \from/\to in {23/13,23/24,13/14,24/14,67/57}
    \draw (\from) -- (\to);
\end{tikzpicture}
\end{minipage}
\end{center}
Here the involution $\varphi$ fixes $13$ and $24$, swapping $23 \leftrightarrow 14$ and swapping $67 \leftrightarrow 57$.

Theorem~\ref{thm:fold-result} will assert that the two linear orders $<_u,<_v$
in Theorem~\ref{thm:calcuating-maj-rule},
whose intersection gives $\majrule{\voting}$, have their inversions sets $\Inv(u)$ and $\Inv(v)$,
respectively, 
given by the elements of $P_C=\Inv(w)$ lying strictly and weakly ``below the fold,'' respectively.

\begin{example} \rm
\label{ex:intro-example-4}
In the above running example,
the uniform vote tally $\voting = \one_{\prefix(C)}$ 
has $\majrule{\voting}$ equal to $ <_u \cap <_v$
with $u=1324576$ since $\Inv(u)=\{23,67\}$ is the set of elements of $P_C$
lying strictly below the fold,
and with $v=3142576$ since
$\Inv(v)=\{23,13,24,67\}$ is
the set of elements of $P_C$
lying weakly below the fold.
\end{example}

Section~\ref{sec:examples} then applies
Theorem~\ref{thm:fold-result} toward calculating $\majrule{\voting}$
for the uniform voting tally $\voting = \one_{\prefix(C)}$ 
on several well-known Condorcet domains $\prefix(C)$ 
of tiling type, including Fishburn's alternating scheme and some of its generalizations, as well as the commutation classes of the lexicographically first and lexicographically last reduced words for the longest permutation $w_0=(n,n-1,\ldots,3,2,1)$.

\section{Heap posets and the proof of Theorem~\ref{thm:calcuating-maj-rule}} 
\label{sec:heaps}

Condorcet domains of tiling type,
defined in Definition~\ref{def:Condorcet-domain-of-tiling-type}, depend upon the choice of a permutation $w$ in $S_n$ and the choice of a commuting equivalence class $C=C(\ii)$ for a reduced word $\ii$ of $w$.
The description of the associated {\it Cartier-Foata poset} or {\it Viennot heap} $P_C$ was also given there, along with additional labelings of
its elements, one by adjacent transpositions $s_i$ and one by the pairs $ab$ in the inversion set $\Inv(w)$.
For the theory of partial commutation
and heap posets in general, see
Cartier and Foata \cite{CartierFoata}
and Viennot \cite{Viennot}.  For its application to commutation classes of reduced words, see 
Elnitsky \cite{Elnitsky},
Fan and Stembridge \cite[\S 2.2]{FanStembridge},
Green and Losonczy \cite[\S 5]{GreenLosonczy}, and
Stembridge \cite{Stembridge}.
Discussion of the relation to Condorcet domains of tiling type
may be found in the work of Danilov, Karzanov, and Koshevoy \cite{DanilovKarzanovKoshevoy, DanilovKarzanovKoshevoy2021} and
Galambos and Reiner \cite[\S 2, Def.~6]{GalambosReiner}.  For related
applications and discussion, see
Banaian, Chepuri, Gunawan, and Pan~\cite[\S 2.1]{BanaianChepuriGunawanPan}, 
Escobar, Pechenik, Tenner, and Yong~\cite[\S 2]{EscobarPechenikTennerYong},
Goodman and Pollack~\cite{GoodmanPollack},
Labb\'e and Lange~\cite[\S 3.1]{LabbeLange},
 Reiner and Roichman~\cite[\S 5]{ReinerRoichman}, and 
Schilling, Thi\'ery, White, and Williams~\cite[\S 2]{SchillingThieryWhiteWilliams}.

\medskip
For the remainder of this section,
fix a permutation $w$ in $S_n$ and a commuting equivalence class $C=C(\ii)$ for one of its reduced words $\ii$. 
Here is the crucial fact that distills some of the discussion in \cite{CartierFoata, FanStembridge, GreenLosonczy, Stembridge, Viennot}. 

\begin{proposition}
\label{prop:ideals-in-heap-are-prefixes}
Regarding the heap $P_C$ as a partial order on $\Inv(w)$,
the following map is a bijection: 
$$
\begin{array}{rrcl}
\DDD=\prefix(C) & \overset{\Inv}{\longrightarrow} & J(P_C)\\
w' &\longmapsto& \Inv(w').
\end{array}
$$
\end{proposition}

The following property of the heap poset $P_C$ on $\Inv(w)$ will be used in the proof of Theorem~\ref{thm:calcuating-maj-rule}.

\begin{proposition}
\label{prop:inversions-on-same-strand-are-comparable}
Two inversions $ab, cd$ in $\Inv(w)$
with $|\{a,b\} \cap \{c,d\}|=1$
are always comparable in  $P_C$.
\end{proposition}
\begin{proof}
Assume for the sake of contradiction that $ab$ and $cd$
are incomparable in $P_C$. Then if one defines 
$$
I:= (P_C)_{<ab} \cup (P_C)_{<cd} =
\{ef \in \Inv(w) \ : \ ef <_{P_C} ab \text{ or } ef<_P cd\},
$$
one concludes that all three
sets $I$, $I \sqcup \{ab\}$, and $I \sqcup \{ab,cd\}$ are ideals of $P_C$. Consequently, one can choose a total ordering of $\Inv(w)=P_C$
with all three of these subsets
as prefixes --- first list the elements of $I$ in any total order extending their partial order, then $ab$, then $cd$, then the remaining elements of $\Inv(w)$ in any total order extending their partial order. This total order corresponds to a shortest factorization $w=s_{j_1} s_{j_2} \cdots s_{j_\ell}$ for $w$, 
with $(j_1,j_2,\ldots,j_\ell)$ in $C$. This contains, as prefixes, shortest factorizations for
$w'$, $w's_{j_k}$, and $w's_{j_k} s_{j_{k+1}}$, 
where $\Inv(w')=I$.  Then $s_{j_k}$ and $s_{j_{k+1}}$ introduce the inversions $ab$ and $cd$, respectively. Because $|\{a,b\} \cap \{c,d\}|=1$, the adjacent transpositions $s_{j_k}$ and $s_{j_{k+1}}$ do not commute.  Thus $|j_k-j_{k+1}|=1$ and so
$ab <_{P_C} cd$, contradicting the assumption that $ab$ and $cd$ are incomparable.
\end{proof}

Now fix a voting profile  $\voting: S_n \rightarrow \NN$ having  $\supp(\voting)=\DDD$.
Recall from Definition~\ref{def:voting-tally-function} the $\voting$-tally function
$$
\begin{array}{rrcl}
\Sigma_\voting: & \Inv(w) &\longrightarrow &\{1,2,\ldots\}\\[1ex]
 & ab & \longmapsto &\Sigma_\voting(ab):=
\hspace{-.1in}\sum\limits_{\substack{w \in \prefix(C):\\ 
ab \in \Inv(w)}} \hspace{-.1in} \voting(w).
\end{array}
$$

\begin{proposition}
\label{prop:vote-tally-rewritings}
One has these re-interpretations of  $\Sigma_\voting$:
\begin{align}
\label{eq:voting-tally-rewriting-1}
\Sigma_\voting(ab)
&:=
\sum_{\substack{w' \in \prefix(C):\\ 
ab \in \Inv(w')}} \voting(w')
\hspace{.1in} =
\sum_{\substack{I \in J(P_C):\\ ab \in I}}\voting( \Inv^{-1}(I) )\\[1ex]
\label{eq:voting-tally-rewriting-2}
&=
\sum_{\substack{w' \in \prefix(C):\\ 
b <_{w'} a}} \voting(w')
\hspace{.1in} = 
\sum_{\substack{w' \in S_n:\\ b <_{w'} a}} \voting(w').
\end{align}
\end{proposition}
\begin{proof} 
In~\eqref{eq:voting-tally-rewriting-1},
the first equality
is Definition~\ref{def:voting-tally-function}, rewritten with a different summation variable $w'$ so as not to suggest dependence on our fixed $w$, 
and the second equality is
from the bijection in  Proposition~\ref{prop:ideals-in-heap-are-prefixes}.
In~\eqref{eq:voting-tally-rewriting-2},
the first equality is from the definition of $\Inv(w')$, and the second from assuming $\supp(\voting)=\prefix(C)$.
\end{proof}

Proposition~\ref{prop:vote-tally-rewritings} has poset-theoretic consequences.
Indeed, recall that in a poset $P$, 
\begin{itemize} 
\item a subset $A \subseteq P$ is called an {\it antichain} if no 
two elements of $A$ are comparable in $P$, and 
\item a subset $I \subseteq P$ is called an {\it (order) ideal} if $x <_P y$ and $y \in I$
implies $x \in I$.
\end{itemize}

\begin{proposition}
\label{prop:strictly-decreasing}
The function $\Sigma_\voting$ 
is strictly decreasing as one moves up in the poset $P_C$; that is, $ab <_{P_C} cd$ implies $\Sigma_\voting(ab) > \Sigma_\voting(cd)$.

Consequently, one has the following for all $r$ in $\RR$:
\begin{itemize}
    \item[(i)] The level set $A:=\{ab \in \Inv(w): \Sigma_\voting(ab)=r\}$ is an antichain in $P_C$,
    and hence has the form
    $$A=\{a_1b_1, a_2b_2,\ldots,a_mb_m\}$$
    for pairwise disjoint inversions; that is, $\{a_j,b_j\} \cap \{a_k,b_k\}=\varnothing$ for $1 \leq j \neq k \leq m$.
    \item[(ii)] The weak superlevel set $I:=\{ab \in \Inv(w): \Sigma_\voting(ab) \geq r\}$ is an order ideal $I$ in $P_C$, and therefore of the
    form $I=\Inv(v)$ for some $v$ in $\prefix(C)$. 
    \item[(iii)] The strict superlevel set $I':=\{ab \in \Inv(w): \Sigma_\voting(ab) > r\}$ is an order ideal $I'$ in $P_C$, and therefore of the
    form $I':=\Inv(u)$ for some $u$ in $\prefix(C)$.
\end{itemize}
\end{proposition}
\begin{proof}
Note that $ab <_{P_C} cd$ implies
$
\{I \in J(P_C): ab \in I\} \supset 
\{I \in J(P_C): cd \in I\},
$
and the inclusion is strict, since the principal
ideal $I:=(P_C)_{\leq ab}$  contains $ab$ but not $cd$.  As $\supp(\voting)=\prefix(C)$
implies $\voting(\Inv^{-1}(I)) > 0$
for all $I$ in $J(P_C)$, one concludes from~\eqref{eq:voting-tally-rewriting-1} that
$$
\Sigma_\voting(ab)
\hspace{.1in} =\sum_{\substack{I \in J(P_C):\\ ab \in I}}\voting( \Inv^{-1}(I) )
\hspace{.1in} >
\sum_{\substack{I \in J(P_C):\\ cd \in I}}\voting( \Inv^{-1}(I) )
\ = \
\Sigma_\voting(cd).
$$
Assertions (ii) and (iii) follow immediately from $\Sigma_\voting$ being a
decreasing function. The assertion in (i) that $A$ is an antichain  follows from the fact that $\Sigma_\voting$ is {\it strictly} decreasing, and the rest of assertion (i)
follows from Proposition~\ref{prop:inversions-on-same-strand-are-comparable}.
\end{proof}

This allows us to prove Theorem~\ref{thm:calcuating-maj-rule}, whose
statement we recall here.
Recall also that $|\voting|:=\sum_{w \in S_n} \voting(w)$.

\begin{customthm}{\ref{thm:calcuating-maj-rule}}
    Assume $\voting: S_n \rightarrow \NN$
has $\supp(\voting)=\prefix(C)$ of tiling type. Regarding $P_C$ as a poset on $\Inv(w)$, 
the $\voting$-tally function $\Sigma_\voting: P_C \rightarrow \{1,2,\ldots\}$
is a strictly decreasing function,
that is, 
$$
ab <_{P_C} cd \text{ implies } \Sigma_\voting(ab) > \Sigma_\voting(cd).
$$
Furthermore, the majority relation
$\majrule{\voting}$, viewed as a subset of $[n] \times [n]$, is the intersection of two linear orders $<_u$ and $<_v$, with 
$u,v \in\prefix(C)$ defined uniquely by these inversion sets, which are order ideals in $P_C$:
:
\begin{align*}
\Inv(u)&=\left\{ ab \in \Inv(w): \Sigma_\voting(ab)> \frac{1}{2}|\voting|\right\},\\
\Inv(v)&=\left\{ ab \in \Inv(w): \Sigma_\voting(ab) \geq \frac{1}{2}|\voting|\right\}.
\end{align*}
\end{customthm}

\begin{proof}
The first statement was proved in Proposition~\ref{prop:strictly-decreasing}. 
Setting $r:=\frac{1}{2}|\voting|$ in
Proposition~\ref{prop:strictly-decreasing}
implies that these sets 
\begin{align*}
I&:=\{ab \in \Inv(w): \Sigma_\voting(ab) > r\},\\
A&:=\{ab \in \Inv(w): \Sigma_\voting(ab) = r\},
\end{align*}
have $A=\{a_1b_1,a_2b_2,\ldots,a_mb_m\}$ an antichain in $P_C$,
and both $I$ and $I \sqcup A$ are order ideals in $P_C$.  Hence there will exist permutations $u,v$ in $\prefix(C)$
with $\Inv(u)=I$ and $\Inv(v)=I \sqcup A$.
It follows that the two linear orders
$<_u$ and $<_v$ almost agree, and the binary relation $R \subset [n] \times [n]$ which
is the intersection $<_u \cap <_v$ 
is a prelinear order $(B_1,B_2,\ldots,B_{n-m})$ with each block $B_j$
either a singleton $\{c\}$ or a pair $\{a_i,b_i\} \in A$. In other words, the intersection
$<_u \cap <_v$ 
is a prelinear order
with only simple ties.

On the other hand, this binary relation $<_R$
given by $<_u \cap <_v$ has this alternate description:
\begin{align*}
a <_R b \ 
&\text{ if and only if } a<_v b \text{ and } a<_u b,\\
&\text{ if and only if } a<_v b \text{ and } ab \not\in A,\\ 
&\text{ if and only if } \Sigma_\voting(ab) \geq r \text{ but } \Sigma_\voting(ab) \neq r,\\ 
&\text{ if and only if } \Sigma_\voting(ab) > r,\\
&\text{ if and only if } \sum_{w' \in S_n: a <_{w'} b} \hspace{-.1in} \voting(w') > r,\\
&\text{ if and only if } a \majrule{\voting} b. \qedhere
\end{align*}
\end{proof}

Note that, as part of the proof of Theorem~\ref{thm:calcuating-maj-rule},
Corollary~\ref{cor:no big/weird ties} from Section~\ref{sec:intro} has also been proved.

\begin{customcor}{\ref{cor:no big/weird ties}}
    When $\voting:S_n \rightarrow \NN$ has $\supp(\voting)$ that is a Condorcet domain of tiling type, then its majority relation $\majrule{\voting}$ is a prelinear order with only simple ties. 
\end{customcor}

As mentioned in Section~\ref{sec:intro}, the remainder of this paper will focus on computing the majority relation $\majrule{\voting}$ for the uniform vote tally $\voting = \one_{\prefix(C)}$  
on a Condorcet domain $\prefix(C)$ 
of tiling type.  The following observation can be useful in this context.

\begin{proposition}
\label{prop:uniform-vote-tally-function}
    For the uniform vote tally $\voting = \one_{\prefix(C)}$ 
    on a Condorcet domain $\prefix(C)$ 
    of tiling type, the $\voting$-tally function
    $\Sigma_\voting: P_C \rightarrow \{1,2,\ldots\}$ can be reformulated as follows:  for $ab \in \Inv(w)=P_C$, one has
    $$
    \Sigma_\voting(ab)=|\{I \in J(P_C): ab \in I\}|=|J((P_C)_{\not\leq ab})|
    $$
    where $(P_C)_{\not\leq ab}:=\{cd \in P_C: cd \not\leq_{P_C} ab\}$.
\end{proposition}
\begin{proof}
The first equality is~\eqref{eq:voting-tally-rewriting-1},
and the second comes from the bijection 
that sends $I \mapsto (I \setminus (P_C)_{\leq ab})$. 
\end{proof}

\section{Heaps with a horizontal folding symmetry}
\label{sec:horizontal-folds}

The goal of this section is Theorem~\ref{thm:fold-result} below, vastly simplifying
the calculation of $\majrule{\voting}$
for Condorcet domains $\prefix(C)$ of tiling type whose heap posets $P_C$ have the following symmetry.

\begin{definition} \rm
\label{def:fold-symmetry}
    Given a finite poset $P$, a self-map $\varphi: P \rightarrow P$ is an {\it antiautomorphism} if it is bijective and $x <_P y$ if and only if $\varphi(x) >_P \varphi(y)$.
Call $\varphi$ a {\it horizontal folding symmetry} if it is an 
antiautomorphism of $P$, which is an involution ($\varphi^2=1_P$), and there exists an order ideal $I$ of $P$ for which
    \begin{itemize}
    \item[(i)] $x \leq_P \varphi(x)$ for all $x \in I$,
    \item[(ii)] $I \cup \varphi(I)=P$,
    \item[(iii)] $I \cap \varphi(I)=:A$ forms an antichain in $P$.
    \end{itemize}
\end{definition}

Informally, we think of the elements of $I$, $\varphi(I)$, and $A$ in Definition~\ref{def:fold-symmetry} as lying weakly {\it below}, {\it above}, and {\it on the fold}, respectively.

\begin{example}\rm
As noted in the introduction,
the poset from Examples~\ref{ex:intro-example}, \ref{ex:intro-example-2}, \ref{ex:intro-example-3}, \ref{ex:intro-example-4}, and~\ref{ex:intro-example-5} has such a horizontal folding symmetry
$\varphi$ that fixes $13$ and $24$, swaps $23 \leftrightarrow 14$, and swaps $67 \leftrightarrow 57$.
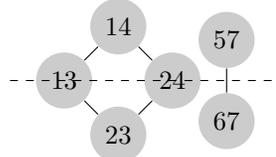
\begin{center}
\begin{minipage}{.3\textwidth}
 \hspace*{0.25\linewidth}
\begin{tikzpicture}
  [scale=.18,auto=left,every node/.style={circle,fill=black!20}]
  \node (23) at (0,0) {$23$};
  \node (13) at (-4,4) {$13$};
  \node (24) at (4,4) {$24$};
  \node (14) at (0,8) {$14$};
  \node (67) at (8,1) {$67$};
  \node (57) at (8,7) {$57$};
\coordinate (leftfoldend) at (-8,4);
\coordinate (rightfoldend) at (12,4);
\draw[dashed] (leftfoldend) -- (rightfoldend);
  \foreach \from/\to in {23/13,23/24,13/14,24/14,67/57}
    \draw (\from) -- (\to);
\end{tikzpicture}
\end{minipage}
\end{center}
Here the imaginary {\it fold} is indicated by a dashed line, and
\begin{align*}
I&=\{23,13,24,67\} \text{ lie weakly below the fold}\\ 
\varphi(I)&=\{13,14,24,57\} \text{ lie weakly above the fold},\\
A&=\{13,24\} \text{ lie on the fold}.
\end{align*}
\end{example}

Motivated by Proposition~\ref{prop:uniform-vote-tally-function}, for a finite poset $P$, define the strictly decreasing function
$\Sigma: P \rightarrow \{1,2,\ldots\}$
\begin{equation}
\label{eq:ideal-tally-rephrasing}
\Sigma(x):=|\{I \in J(P):x \in I\}|=|J(P_{\not\leq x})|.
\end{equation}

\begin{lemma}
\label{lem:fold-balancing}
When $P$ has an antiautomorphism $\varphi$, every $x$ in $P$ satisfies
$
\Sigma(x) + \Sigma(\varphi(x))= |J(P)|.
$
\end{lemma}

\begin{proof}
One has the following equalities, with justifications below:
\begin{eqnarray*}
|J(P)|
&=&|\{I \in J(P): x \in I\}|+ |\{I \in J(P): x \not\in I\}|,\\
&\overset{(a)}{=}&\Sigma(x)+ 
|J(P \setminus P_{\geq x})|,\\
&\overset{(b)}{=}&\Sigma(x)+ |J(P \setminus P_{\leq \varphi(x)})|,\\
&\overset{(c)}{=}&\Sigma(x)+ \Sigma(\varphi(x)).
\end{eqnarray*}
Equality (a) uses the fact that an ideal $I \subset P$ does not contain $x$ if and only if $I \subseteq P \setminus P_{\geq x}$. 
Equality (b) comes from two facts. First, the map 
$\varphi$ restricts to a poset anti-isomorphism between
$P \setminus P_{\geq x}$ and $P \setminus P_{\leq \varphi(x)}$,
so that $P \setminus P_{\leq \varphi(x)}$
is isomorphic to the {\it opposite} or {\it dual} poset $(P \setminus P_{\geq x})^{\mathrm{opp}}$.  Second,
for any poset $Q$, the map $J(Q^{\mathrm{opp}}) \rightarrow J(Q)$ sending $I \mapsto Q \setminus I$ is a
bijection, so that $|J(Q^{\mathrm{opp}})|=|J(Q)|$.
Equality (c) uses~\eqref{eq:ideal-tally-rephrasing}.
\end{proof}

\begin{corollary}
\label{cor:fold-corollary}
When $P$ has a horizontal folding symmetry $\varphi$,
\begin{itemize}
\item[(i)]
every $x$ in $I$ weakly below the fold has 
$
\Sigma(x) \geq \frac{1}{2}|J(P)|
$,
with equality if and only if $x \in A$, and
\item[(ii)]
every $y$ in $\varphi(I)$ weakly above the fold has 
$\frac{1}{2}|J(P)| \geq \Sigma(y)$,
with equality if and only if $y \in A$.
\end{itemize}
\end{corollary}
\begin{proof}
Lemma~\ref{lem:fold-balancing} shows $\Sigma(x)+\Sigma(\varphi(x))=|J(P)|$.
The poset relation $x \leq_P \varphi(x)$ for $x \in I$ from Definition~\ref{def:fold-symmetry}(i), along with the fact that $\Sigma$ is a decreasing function gives  the weak
    inequality $\Sigma(x) \geq \frac{1}{2}|J(P)|$.  The weak inequality $\frac{1}{2}|J(P)| \geq \Sigma(y)$ is argued similarly, using the fact that $\varphi^2=1$ and $\varphi(y) \in I$.

When $x \in A = I \cap \varphi(I)$, the above implies 
    both inequalities $\Sigma(x) \geq \frac{1}{2}|J(P)| \geq \Sigma(x)$, 
    so $\Sigma(x)=\frac{1}{2}|J(P)|$.  Conversely, if one has equality $\Sigma(x) = \frac{1}{2}|J(P)|$, then Lemma~\ref{lem:fold-balancing} shows that one must have also $\Sigma(\varphi(x)) = \frac{1}{2}|J(P)|$.  But then the poset relation $x \leq_P \varphi(x)$
    along with the {\it strictly} order-reversing nature of $\Sigma$ implies $x=\varphi(x)$.  Hence $x$ lies in $I \cap \varphi(I)=A$. 
\end{proof}

Applying Corollary~\ref{cor:fold-corollary}
and Theorem~\ref{thm:calcuating-maj-rule}
to Condorcet domains of tiling type gives the following.

\begin{theorem}
\label{thm:fold-result}
Let $\prefix(C)$  
be a Condorcet domain of tiling type whose heap poset $P_C$ has a horizontal folding symmetry $\varphi$, with  $I$ the ideal weakly below the fold, and $A=I \cap \varphi(I)$.
Then the uniform vote tally $\voting = \one_{\prefix(C)}$ 
has as majority relation $\majrule{\voting}$ the intersection $<_u \cap <_v$ where
$\Inv(u)=I \setminus A$ and $\Inv(v)=I$.
\end{theorem}

Example~\ref{ex:intro-example-4} illustrated Theorem~\ref{thm:fold-result} for the running example Section~\ref{sec:intro}.  The next section applies Theorem~\ref{thm:fold-result} to
several well-known Condorcet domains of tiling type.

\section{Examples}
\label{sec:examples}

In this section, we apply Theorem~\ref{thm:fold-result} to compute, explicitly, the majority relation $\majrule{\voting}$
for the uniform vote tally $\voting = \one_{\prefix(C)}$ 
on several well-known Condorcet domains $\prefix(C)$ 
of tiling type.  In fact,
some of these will be {\it maximal} Condorcet domains of tiling type, meaning that $C=C(\ii)$ is a commutation class for a reduced word $\ii$ of the {\it longest permutation} $w_0=(n,n-1,\ldots,3,2,1)$.  Such commutation classes also appear in the literature under the guise of 
\begin{itemize}
   \item {\it reflection networks} (see Knuth \cite[\S 8]{Knuth}), 
   \item {\it sorting networks} (see Angel, Holroyd, Romik, and Virag~\cite{AngelHolroydRomikVirag}), and
\item the {\it higher Bruhat order} $B(n,2)$ (see Manin and Schechtman~\cite{ManinSchechtman}, Ziegler~\cite{Ziegler}.
\end{itemize}

\subsection{Totally ordered heap posets and cocktail-shaker sorts}

We begin with an application of Theorem~\ref{thm:fold-result} to a particularly restrictive category of commutation class; namely, when the class $C(\ii)$ consists of a single element: $C(\ii) = \{\ii\}$. 

These singleton commutation classes were characterized by Tenner in \cite{Tenner-OneElement}. In the language of Section~\ref{sec:heaps}, they are exactly the classes $C := C(\ii)$ whose heap posets $P_C$ are totally ordered. And of course, each element $w' \in \prefix(C)$ has this same property; equivalently, each order ideal in the totally ordered poset $P_C$ is also a totally ordered poset.

\begin{proposition}\label{prop:one-element class}
    Let $w$ be a permutation with reduced word $\ii = (i_1,\ldots, i_{\ell})$ such that $C:= C(\ii) = \{\ii\}$. Then the uniform vote tally $\voting = \one_{\prefix(C)}$ 
    on the Condorcet domain $\prefix(C)$ 
    has majority relation $\majrule{\voting}$ equal to $<_u \cap <_v$, where $u = s_{i_1}s_{i_2}\cdots s_{i_{\lfloor \ell/2\rfloor}}$ and $v=s_{i_1}s_{i_2}\cdots s_{i_{\lceil \ell/2\rceil}}$. In particular,
    \begin{itemize}
        \item if $\ell$ is even, then $\majrule{\voting}$ is $<_u$ for $u=s_{i_1}s_{i_2}\cdots s_{i_{\ell/2}}=v$
        and 
        \item if $\ell$ is odd, then $\majrule{\voting}$ is $<_u \cap <_v$ where $u=s_{i_1}s_{i_2}\cdots s_{i_{(\ell-1)/2}}$ and $v=s_{i_1}s_{i_2}\cdots s_{i_{(\ell+1)/2}}$.
    \end{itemize}
\end{proposition}

\begin{proof}
The heap poset $P_C$ is a chain with $\ell$ elements. Write $w^{(t)} := s_{i_1} \cdots s_{i_t}$ for $1 \le t \le \ell$, so $\prefix(C) = \{w^{(0)}:= e\,,\, w^{(1)}\,,\, \ldots\,,\, w^{(\ell)} = w\}$. We now define a horizontal folding symmetry $\varphi$ on $P_C$, and it will be easiest to do so if we label the elements of $P_C$ by adjacent transpositions. With that labeling, define $\varphi(s_{i_t}) := s_{i_{\ell-t}}$. The ideal $I$ weakly below the fold is then, necessarily, the set $\{s_{i_t} : t \le \ell/2\}$. Hence, by Theorem~\ref{thm:fold-result}, the uniform vote tally has as majority relation $\majrule{\voting}$ the intersection $<_u \cap <_v$ where $u = w^{(\lfloor \ell/2\rfloor)}$ and $v = w^{(\lceil \ell/2\rceil)}$.
\end{proof}

The previous result holds for any single-element class $C(\ii)$, with no further restrictions on the permutation $w$. One interesting family of examples, when $w$ is the longest permutation $w_0 = (n,n-1,\ldots,3,2,1)$ arises from the \emph{cocktail-shaker sort}.

\begin{example}\label{ex:cocktail-shaker} \rm
The \emph{cocktail-shaker sort}, discussed by Knuth in \cite[\S 8]{Knuth}, is a sorting algorithm that first moves the largest element to the correct position, then the smallest, then the second largest, then the second smallest, and so on. This procedure, and its symmetric analogues, yields the four reduced words for $w_0 = (n,n-1,\ldots, 3,2,1)$ that have one-element commutation classes, as described in \cite{Tenner-OneElement}. Let $n = 6$, and consider $\ii = (5,4,3,2,1,2,3,4,5,4,3,2,3,4,3)$.
The length of $\ii$ is $15$, so, by Proposition~\ref{prop:one-element class}, the majority relation $\majrule{\voting}$ on $\prefix(C(\ii))$ is $<_u \cap <_v$, where
$$
    u=s_5s_4s_3s_2s_1s_2s_3 = 623145 \text{ \ and \ }
    v= us_4 = 623415.
$$
Thus we can abbreviate this prelinear order as 
$$6\,2\,3\,\{1\,4\}\,5.$$
On the other hand, if we use $\jj = (3,4,3,2,3,4,5,4,3,2,1,2,3,4,5)$, the reverse of $\ii$ and another reduced word that is its own commutation class, then the majority relation $\majrule{\voting}$ on $\prefix(C(\jj))$ is $<_u \cap <_v$, where
$$u = s_3s_4s_3s_2s_3s_4s_5 = 154362 \text{ \ and \ } v = us_4 = 154632.$$
Thus we can abbreviate this prelinear order as 
$$1\,5\,4\,\{3\,6\}\,2.$$
\end{example}

In contrast to Example~\ref{ex:cocktail-shaker}, we note that a reduced word arising from the cocktail-shaker sort of an arbitrary permutation (that is, not $w_0$) might not be in a one-element commutation class. For example, the cocktail-shaker sort of the permutation $2143 \in S_4$ produces $\ii = (1,3)$, and $C(\ii) = \{(1,3),(3,1)\}$.

\subsection{The lex-first commutation class}

Theorem~\ref{thm:fold-result} can be applied to other commutation classes for the longest permutation $w_0$, as well, not only the few addressed by Proposition~\ref{prop:one-element class}. Here we consider the class containing the lexicographically minimal reduced word for $w_0$ in $S_n$:
\begin{align*}
\lf := (&1,\\
&2,\ 1,\\
&3,\ 2,\ 1,\\
&\vdots\\
&n-1,\ \ldots,\ 3,\ 2,\ 1).
\end{align*}

\begin{proposition}
\label{prop:lex-least word for long element}
    Let $w_0 = (n,n-1, \ldots, 3,2,1)$ be the longest permutation in $S_n$, define $\lf$ as above, and set $C:= C(\lf)$. 
    Then the uniform vote tally $\voting = \one_{\prefix(C)}$ 
    on the Condorcet domain $\prefix(C)$ 
    has majority relation $\majrule{\voting}$ equal to 
    \begin{align*}
        &\left\{\left(\frac{n}{2}\right)\  \left(\frac{n}{2} + 1\right)\right\} \ \cdots\ \{3\ (n-2)\}\ \{2\ (n-1)\}\ \{1\ n\}& & \text{if $n$ is even,}\\
        &\left(\frac{n+1}{2}\right)\ \left\{\left(\frac{n-1}{2}\right) \ \left(\frac{n+3}{2}\right)\right\} \ \cdots\ \{3 \ (n-2)\}\ \{2 \ (n-1)\}\ \{1 \ n\}& & \text{if $n$ is odd.}
    \end{align*}
\end{proposition}

\begin{proof}
    The heap poset $P_C$ has the shape of a rotated staircase. 
    It will be easiest, in this case, to label the elements of $P_C$ by the unique inversion pair introduced by that element in the poset. 
    At each rank of $P_C$, the sum $a+b$ for all inversions $ab$ at that rank is fixed. Moreover, all possible pairs $ab$ for which $1 \le a < b \le n$ appear, and $ab$ is covered by $(a+1)b$ and $a(b+1)$ when these are valid inversion pairs. For example, when $n = 5$, this poset is as follows, with the fold indicated by a dashed line.
    $$
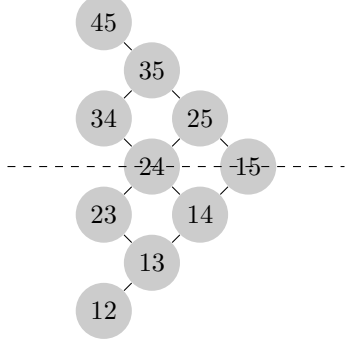
\begin{tikzpicture}
        [scale=.16,auto=left,every node/.style={circle,fill=black!20}]
        \node (12) at (0,0) {$12$};
        \node (13) at (4,4) {$13$};
        \node (14) at (8,8) {$14$};
        \node (15) at (12,12) {$15$};
        \node (25) at (8,16) {$25$};
        \node (35) at (4,20) {$35$};
        \node (45) at (0,24) {$45$};
        \node (23) at (0,8) {$23$};
        \node (24) at (4,12) {$24$};
        \node (34) at (0,16) {$34$};
        \foreach \from/\to in {12/13,13/14,14/15,
        15/25,25/35,35/45,
        23/24,24/25,34/35,
        23/13,14/24,24/34}
        \draw (\from) -- (\to);
        \coordinate (leftfoldend) at (-8,12);
\coordinate (rightfoldend) at (20,12);
\draw[dashed] (leftfoldend) -- (rightfoldend);
    \end{tikzpicture}$$

    The poset $P_C$ has a horizontal folding symmetry $\varphi$ defined by $\varphi(ab) = (n+1-b)(n+1-a)$. 
    The ideal $I$ weakly below the fold in this symmetry consists of all elements $ab$ with $a+b \le n+1$, and the antichain $A \subset I$ is the set of elements $ab$ with $a+b = n+1$. Hence, by Theorem~\ref{thm:fold-result}, the uniform vote tally has majority relation $\majrule{\voting}$, the intersection $<_u \cap <_v $ with 
    $$u = \begin{cases}
        \left(\frac{n}{2}\right)\left(\frac{n}{2} + 1\right)\, \cdots \,3\,(n-2)\,2\,(n-1)\,1\,n& \text{if $n$ is even, and}\\
        \left(\frac{n+1}{2}\right)\left(\frac{n-1}{2}\right)\left(\frac{n+3}{2}\right) \,\cdots\, 3\,(n-2)\,2\,(n-1)\,1\,n& \text{if $n$ is odd;}
    \end{cases}$$
    $$v = \begin{cases}
        \left(\frac{n}{2} + 1\right)\left(\frac{n}{2}\right) \,\cdots\, (n-2)\,3\,(n-1)\,2\,n\,1& \text{if $n$ is even, and}\\
        \left(\frac{n+1}{2}\right)\left(\frac{n+3}{2}\right)\left(\frac{n-1}{2}\right) \,\cdots \,(n-2)\,3\,(n-1)\,2\,n\,1& \text{if $n$ is odd.}
    \end{cases}$$
    Note that, whatever the parity of $n$, the permutation $v$ is obtained from $u$ by swapping the rightmost two letters, the next rightmost two letters, and so on. In other words, the majority relation $<_u \cap <_v$ has $\lfloor n/2 \rfloor$ simple ties, between values $a$ and $n+1-a$, and these (including a ``trivial'' tie when $n$ is odd) appear from left to right in increasing order of their largest values.
\end{proof}

\subsection{Fishburn's alternating scheme, or the odd-even sort}

In this section, we explore the majority relation for commutation classes of reduced words for permutations related to powers of {\it bipartite Coxeter elements}. This is a family of commutation classes, each for a different permutation, one of which is the longest permutation $w_0$. To set up that work, fix an integer $n$, let $k:=\lfloor \frac{n-1}{2} \rfloor$, and define
\begin{align*}
\ceven&:=s_2 s_4 \cdots s_{2k},\\
\codd&:=s_1 s_3 \cdots s_{2k\pm 1}, \text{ and}\\
c\ &:=\codd \ceven=s_1 s_3 \cdots s_{2k\pm 1} \cdot s_2 s_4 \cdots s_{2k},
\end{align*}
so that $c$ is a product of all of the adjacent transpositions
$s_1,s_2,\ldots,s_{n-1}$ in a particular order, making $c$ an example of {\it Coxeter element} in $S_n$; see Humphreys \cite[\S 3.16]{Humphreys}. 
The particular Coxeter elements
of the form $c=\codd \ceven$
and $c^{-1}=\ceven \codd$
are often called {\it bipartite Coxeter elements}; see Reading \cite{Reading}.

The reduced words defined by these products are related to Fishburn's alternating scheme \cite{Fishburn1997, Fishburn2002} and Knuth's odd-even sort \cite[\S 8]{Knuth}. The permutations discussed below also appear in recent work of Banaian, Chepuri, Gunawan, Pan \cite[\S 7.2]{BanaianChepuriGunawanPan}, and this is related to Labb\'e and Lange's \cite[Example 4.1]{LabbeLange}.

\begin{proposition}\label{prop:fishburn alternating schemes for oddly many factors}
Fix positive integers $n$ and $p$ with $p \leq (n-1)/2$, and consider the permutation $w$ in $S_n$ given by
$$
w = c^p \codd = (\codd \ceven)^p \codd
=\underbrace{\codd \ceven \codd \ceven \cdots \codd \ceven \codd}_{2p+1 \text{ factors}},
$$
with $c$ and $\codd$ as defined above. Let $\ii$ be the reduced word defined by that product, and set $C := C(\ii)$.  
In particular, when $n$ is odd and $p=(n-1)/2$, this $w=c^{\frac{n-1}{2}}\codd=w_0$,
and $\prefix(C)$ is Fishburn's alternating scheme for odd $n$.

Then the uniform vote tally $\voting = \one_{\prefix(C)}$ 
has majority relation $\majrule{\voting}$ equal to $<_u \cap <_v$ where the unordered pair\footnote{More explicitly, this means $(u,v)=(c^{p/2}, c^{p/2} \codd)$ for $p$ even, and $(u,v)=(c^{(p-1)/2} \codd, c^{(p+1)/2})$ for $p$ odd.}
    $\{u,v\}=\{c^{\lceil p/2 \rceil}, c^{\lfloor p/2\rfloor}\codd\}$.
\end{proposition}

\begin{proof}
    As in the previous examples, we consider the heap poset $P_C$ and apply Theorem~\ref{thm:fold-result}. In this case, it is most appealing to label the heap elements by adjacent transpositions. 
    The poset $P_C$ is self-dual, and it has rank $2p$.  
It has a horizontal folding symmetry $\varphi$ that maps the element $s_k$ at rank $t$ to the element $s_k$ at rank $2p-t$.
The heap poset $P_C$ is shown below, for $n=10$ and $p=3$, with the fold indicated by a dashed line.

    $$
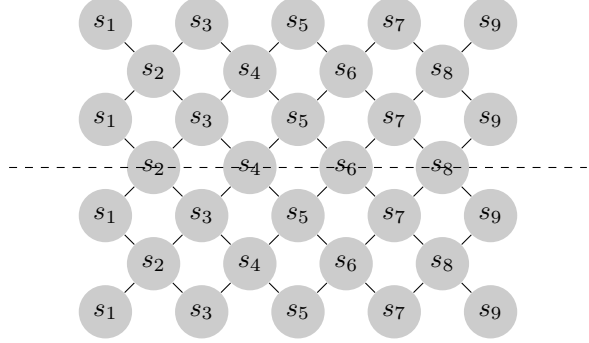
\begin{tikzpicture}
        [scale=.16,auto=left,every node/.style={circle,fill=black!20}]
        \node (00) at (0,0) {$s_1$};
        \node (20) at (8,0) {$s_3$};
        \node (40) at (16,0) {$s_5$};
        \node (60) at (24,0) {$s_7$};
        \node (80) at (32,0) {$s_9$};
        \node (11) at (4,4) {$s_2$};
        \node (31) at (12,4) {$s_4$};
        \node (51) at (20,4) {$s_6$};
        \node (71) at (28,4) {$s_8$};
        \node (02) at (0,8) {$s_1$};
        \node (22) at (8,8) {$s_3$};
        \node (42) at (16,8) {$s_5$};
        \node (62) at (24,8) {$s_7$};
        \node (82) at (32,8) {$s_9$};
        \node (13) at (4,12) {$s_2$};
        \node (33) at (12,12) {$s_4$};
        \node (53) at (20,12) {$s_6$};
        \node (73) at (28,12) {$s_8$};
        \node (04) at (0,16) {$s_1$};
        \node (24) at (8,16) {$s_3$};
        \node (44) at (16,16) {$s_5$};
        \node (64) at (24,16) {$s_7$};
        \node (84) at (32,16) {$s_9$};
        \node (15) at (4,20) {$s_2$};
        \node (35) at (12,20) {$s_4$};
        \node (55) at (20,20) {$s_6$};
        \node (75) at (28,20) {$s_8$};
        \node (06) at (0,24) {$s_1$};
        \node (26) at (8,24) {$s_3$};
        \node (46) at (16,24) {$s_5$};
        \node (66) at (24,24) {$s_7$};
        \node (86) at (32,24) {$s_9$};
        \foreach \from/\to in {00/11, 20/11, 20/31, 40/31, 40/51, 60/51, 60/71, 80/71, 11/02, 11/22, 31/22, 31/42, 51/42, 51/62, 71/62, 71/82, 02/13, 22/13, 22/33, 42/33, 42/53, 62/53, 62/73, 82/73, 13/04, 13/24, 33/24, 33/44, 53/44, 53/64, 73/64, 73/84, 04/15, 24/15, 24/35, 44/35, 44/55, 64/55, 64/75, 84/75, 15/06, 15/26, 35/26, 35/46, 55/46, 55/66, 75/66, 75/86}
        \draw (\from) -- (\to);
        \coordinate (leftfoldend) at (-8,12);
\coordinate (rightfoldend) at (40,12);
\draw[dashed] (leftfoldend) -- (rightfoldend);
    \end{tikzpicture}$$
The ideal $I$ weakly below the fold is the set of elements having rank less than or equal to $p$, and $A$ is the set of elements with rank $p$.

Thinking of $P_C$ as labeled by the inversions in $\Inv(w)$, the permutation $w'$ having
    $\Inv(w')$ equal to the set of inversions at ranks $0, 1, 2, \ldots, t$ is $c^{t/2}\codd$ when $t$ is even, and $c^{(t+1)/2}$ when $t$ is odd. Hence, by Theorem~\ref{thm:fold-result}, the uniform vote tally has majority relation $\majrule{\voting}$ equal to $<_u \cap <_v$, with $u$ and $v$ as claimed.
\end{proof}

To give a general formula for the majority relations described by Proposition~\ref{prop:fishburn alternating schemes for oddly many factors}, let $\{m,q\} = \{n-1,n\}$ so that $m$ is even. Then the majority relation of Proposition~\ref{prop:fishburn alternating schemes for oddly many factors} is 
\begin{multline}\label{eqn:fishburn oddly many terms}
    \ \ \cdots \ \{6 \ (2p-4)\} \ \{4 \ (2p-2)\} \ \{2 \ (2p)\} \\ 
    \{1 \ (2p+2)\} \ \{3 \ (2p+4)\} \ \{5 \ (2p+6)\} \ \cdots \ \{(m-2p-1) \ m\}\\ 
    \{(m-2p+1) \ q\} \ \{(m-2p+3) \ (q-2)\} \ \{(m-2p+5) \ (q-4)\} \ \cdots, \ \ 
\end{multline}
where $\{x \ x\}$ should be interpreted simply as ``$x$.'' The middle row of~\eqref{eqn:fishburn oddly many terms} gives the ``pattern'' of the relation (ties between odd $t$ and $t+2p+1$), while the left and right tails do their best to mimic the pattern for small even values (first row) and large odd values (third row). We compute the majority relation explicitly below, when $n = 16$ (and thus $(m,q) = (16,15)$) and $p \le 5$.
        \begingroup 
        \setlength{\tabcolsep}{20pt}
        \renewcommand{\arraystretch}{1.5}
        $$\begin{array}{|l|l|}
        \hline
        w & \majrule{\voting}\\
        \hline\hline
        \codd & \{1\ 2\}\ \{3\ 4\} \ \{5\ 6\}\  \{7\ 8\}\ \{9\ 10\}\ \{11\ 12\}\ \{13\ 14\}\ \{15\ 16\}\\
        \hline
        c^1\codd & 2\ \{1 \ 4\}\ \{3 \ 6\}\ \{5 \ 8\}\ \{7 \ 10\}\ \{9 \ 12\}\ \{11 \ 14\}\ \{13 \ 16\}\ 15 \\
        \hline
        c^2\codd & \{2 \ 4\} \ \{1 \ 6\} \ \{3 \ 8\} \ \{5 \ 10\} \ \{7 \ 12\} \ \{9 \ 14\} \ \{11 \ 16\} \ \{13 \ 15\} \\
        \hline
        c^3\codd & 4 \ \{2 \ 6\} \ \{1 \ 8\} \ \{3 \ 10\} \ \{5 \ 12\} \ \{7 \ 14\} \ \{9 \ 16\} \ \{11 \ 15\} \ 13 \\
        \hline
        c^4\codd & \{4 \ 6\} \ \{2 \ 8\} \ \{1 \ 10\} \ \{3 \ 12\} \ \{5 \ 14\} \ \{7 \ 16\} \ \{9 \ 15\} \ \{11 \ 13\}\\
        \hline
        c^5\codd & 6 \ \{4 \ 8\} \ \{2 \ 10\} \ \{1 \ 12\} \ \{3 \ 14\} \ \{5 \ 16\} \ \{7 \ 15\} \ \{9 \ 13\} \ 11\\
        \hline
        \end{array}$$
        \endgroup

In contrast to the reduced words defined by $c^p \codd$, the words defined by $c^p$ do not, in general, lead to heaps $P_C$ for which there is a horizontal folding symmetry.  Based on experimental evidence, we make the following conjecture for the majority relation in those cases.

\begin{conjecture}
\label{conj:here-be-dragons}
    Fix positive integers $n$ and $p$ with $p \le n/2$, and consider the permutation $w \in S_n$ given by 
\begin{equation}
\label{eq:dragon-powers-of-c}
w = c^p = (\codd\ceven)^p
    =\underbrace{\codd \ceven \codd \ceven \cdots \codd \ceven \codd}_{2p \text{ factors}},
\end{equation}
    with $c$ as above. Let $\ii$ be the reduced word defined by that product, and set $C := C(\ii)$. 
    In particular, when $n$ is even and $p=n/2$, this $w=c^{n/2}=w_0$,
and $\prefix(C)$ is Fishburn's alternating scheme for even $n$.

    Then the uniform vote tally $\voting = \one_{\prefix(C)}$ has majority relation $\majrule{\voting}$ equal to
    a total order $<_u$ where
    $u$ is the product of the leftmost $p$ factors in \eqref{eq:dragon-powers-of-c}; that is,
    $$u=\begin{cases}
        c^{(p-1)/2} \codd & \text{if $p$ is odd, and}\\
        c^{p/2} & \text{if $p$ is even.}
    \end{cases}$$
\end{conjecture}

\noindent
We suspect this can be resolved by brute force computation of the entire tally function $\Sigma: P_C \rightarrow \{1,2,\ldots\}$.

\subsection{Diamond-shaped grid heaps}

Another family of heap posets where Theorem~\ref{thm:fold-result} is applicable consists of the diamond-shaped Cartesian products of two chains of equal length.

For a positive integer $k$, consider this reduced word $\ii_k$ for a permutation $w$ in $S_{2k}$:
\begin{align*}
\ii_k := (&k,\\
&k-1,\ k+1,\\
&k-2,\ k,\ k+2,\\
& \vdots \\
&2,\ 4,\ 6,\ 8, \ 10, \ \ldots,\ 2k-2, \\
&1,\ 3,\ 5,\ 7, \, 9, \ 11, \ \ldots, \ 2k-3, \ 2k-1, \\
&2,\ 4,\ 6,\ 8, \ 10, \ \ldots,\ 2k-2, \\
&\vdots\\
&k-2,\ k,\ k+2,\\
&k-1,\ k+1,\\
&k).
\end{align*}
For example, the heap poset $P_{C(\ii_4)}$ is as follows, with the fold indicated by a dashed line.
    $$
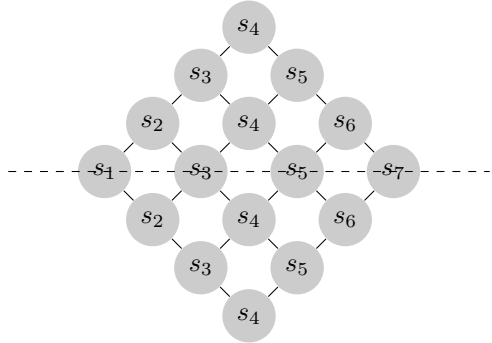
\begin{tikzpicture}
        [scale=.16,auto=left,every node/.style={circle,fill=black!20}]
        \node (00) at (0,0) {$s_4$};
        \node (-11) at (-4,4) {$s_3$};
        \node (11) at (4,4) {$s_5$};
        \node (-22) at (-8,8) {$s_2$};
        \node (02) at (0,8) {$s_4$};
        \node (22) at (8,8) {$s_6$};
        \node (-33) at (-12,12) {$s_1$};
        \node (-13) at (-4,12) {$s_3$};
        \node (13) at (4,12) {$s_5$};
        \node (33) at (12,12) {$s_7$};
        \node (-24) at (-8,16) {$s_2$};
        \node (04) at (0,16) {$s_4$};
        \node (24) at (8,16) {$s_6$};
        \node (-15) at (-4,20) {$s_3$};
        \node (15) at (4,20) {$s_5$};
        \node (06) at (0,24) {$s_4$};
        \foreach \from/\to in {00/-11, -11/-22, -22/-33, 11/02, 02/-13, -13/-24, 22/13, 13/04, 04/-15, 33/24, 24/15, 15/06, 
        00/11, 11/22, 22/33, -11/02, 02/13, 13/24, -22/-13, -13/04, 04/15, -33/-24, -24/-15, -15/06}
        \draw (\from) -- (\to);
        \coordinate (leftfoldend) at (-20,12);
\coordinate (rightfoldend) at (20,12);
\draw[dashed] (leftfoldend) -- (rightfoldend);
    \end{tikzpicture}$$
Note that when $k > 1$, the permutation $w$ in $S_{2k}$ factored by $\ii_k$ is not the longest permutation $w_0$.

\begin{proposition}\label{prop:diamond-shaped heaps}
    Fix a positive integer $k$, define $\ii_k$ as above, and set $C := C(\ii_k)$. Then the uniform vote tally $\voting = \one_{\prefix(C)}$     on the Condorcet domain $\prefix(C)$ 
    has majority relation
    $$\{1 \ (k+1)\} \ \{ 2 \ (k+2)\} \ \{3 \ (k+3)\} \ \cdots \ \{k \ (2k)\}.$$
\end{proposition}

\begin{proof}
    The heap poset $P_C$ has the shape of a diamond, the Cartesian product of two $k$-element chains. It is easiest, in this case, to label the elements of $P_C$ by adjacent transpositions. The poset $P_C$ is self-dual and it has rank $2k-2$. 
    
    Define the horizontal folding symmetry $\varphi$ so that it maps the element $s_k$ at rank $t$ to the element $s_k$ at rank $2k-2-t$. The ideal $I$ weakly below the fold is the set of elements having rank less than or equal to $k-1$, and $A$ is the set of elements of rank $k-1$. Hence, by Theorem~\ref{thm:fold-result}, the uniform vote tally has majority relation equal to $<_u \cap <_v$, where $u$ has reduced word
    $(k,\ k-1,\ k+1,\ k-2,\ k,\ k+2,\ \ldots, \ 2,\ 4,\ 6,\ \ldots,\ 2k-2)$, and $v$ has reduced word $(k,\ k-1,\ k+1,\ k-2,\ k,\ k+2,\ \ldots, \ 2,\ 4,\ 6,\ \ldots,\ 2k-2, \ 1,\ 3,\ 5,\ \ldots,\ 2k-1)$. In other words, omitting commas, we have
    \begin{align*}
        u &= 1 \, (k+1) \, 2 \, (k+2) \, 3 \, (k+3) \, \cdots \, (2k-1) \, k \, (2k), \text{ and}\\
        v &= (k+1) \, 1 \, (k+2) \, 2 \, (k+3) \, 3 \, \cdots \, (k-1) \, (2k) \, k,
    \end{align*}
    and $<_u \cap <_v$ gives the claimed result.
\end{proof}

Continuing the example drawn above, when $k=4$, the majority relation is this prelinear order: 
$$\{1 \ 5\} \ \{2 \ 6\} \ \{3 \ 7\} \ \{4 \ 8\}.$$

\section{Reduction to indecomposable permutations}
\label{sec:disconnected-heaps}

This section points out an easy reduction in analyzing $\majrule{\voting}$ for the uniform vote tally $\voting=\one_{\prefix(C)}$
on a Condorcet domain $\prefix(C)$ of tiling type:
one need only consider {\it indecomposable permutations}, defined here.

\begin{definition} \rm
Given permutations $w^{(1)} \in S_{n_1}$ and $w^{(2)} \in S_{n_2}$, the \emph{direct sum} $w^{(1)} \oplus w^{(2)} \in S_{n_1+n_2}$ is
$$
w^{(1)} \oplus w^{(2)}:=
(w^{(1)}_1,w^{(1)}_2,\ldots,w^{(1)}_{n_1},
n_1+w^{(2)}_1,n_1+w^{(2)}_2,\ldots,n_1+w^{(2)}_{n_2}) 
$$
Call a permutation $w$ in $S_n$ {\it decomposable} if it has this form
$w=w^{(1)} \oplus w^{(2)}$ for some decomposition $n=n_1+n_2$ with $n_1,n_2 \in \{1,2,\ldots\}$.  In other words, a permutation 
$w$ is decomposable if it lies in a proper {\it Young subroup} $S_{n_1} \times S_{n_2}$, permuting the
two blocks
$\{1,2,\ldots,n_1\}$ and 
$\{n_1+1,n_1+2,\ldots,n_1+n_2\}$
independently.
Otherwise, when $w$ is not decomposable, call it {\it indecomposable}.
\end{definition}

Decomposability of $w$ is easily detected from any of its
reduced words or commuting classes as follows.

\begin{definition} \rm
For $w$ in $S_n$, define its {\it $S$-support} 
to be the
subset $S(w):=\{s_{i_1},s_{i_2},\ldots,s_{i_\ell}\}$ (without repetitions) of the adjacent transpositions $S:=\{s_1,s_2,\ldots,s_{n-1}\}$ appearing in any of the shortest factorizations 
$w=s_{i_1} s_{i_2} \cdots s_{i_\ell}$. The set $S(w)$ is well-defined, independent of the factorization, by the
theorem of Matsumoto and Tits \cite[Thm.~3.3.1]{BjornerBrenti} mentioned in Section~\ref{sec:intro}:
any two shortest factorizations are connected by a sequence of
moves~\eqref{eq:commuting-relation} and~\eqref{eq:Yang-Baxter-relation}.
\end{definition}

The following result is well-known (see, for example, \cite[Lemma~2.8]{Tenner-Repetition}), but we include a short proof here for the sake of completeness.

\begin{proposition}
A permutation $w$ in $S_n$ is decomposable if and only if $S(w) \subsetneq S$ is a proper subset. 
\end{proposition}
\begin{proof}
Note that when $w = w^{(1)} \oplus w^{(2)}$,
if one chooses shortest factorizations 
$w^{(1)}=s_{1_1} s_{i_2} \cdots s_{i_{\ell}}$ and
$w^{(2)}=s_{j_1} s_{j_2} \cdots s_{j_{\ell'}}$
then one has a shortest factorization
$
w=s_{1_1} s_{i_2} \cdots s_{i_{\ell}} \cdot 
s_{n_1+j_1} s_{n_1+j_2} \cdots s_{n_1+j_{\ell'}},
$
showing that $s_{n_1} \in S \setminus S(w)$.
Conversely, if $s_{n_1} \in S \setminus S(w)$, then $w$ lies in $S_{n_1} \times S_{n_2}$ and hence is decomposable.
\end{proof}

As a consequence, when one factors a decomposable permutation
$w=w^{(1)} \oplus w^{(2)}$ in $S_{n_1} \times S_{n_2}$, any commuting class $C$ of reduced words
for $w$ is the set of all shuffles of two commuting classes $C^{(1)},C^{(2)}$ for reduced words of
$w^{(1)}, w^{(2)}$, respectively.  One has an accompanying decomposition
of the heap poset as a disjoint union 
$P_C = P_{C^{(1)}} \sqcup P_{C^{(2)}}$,
in which one has applied the relabeling $s_i \mapsto s_{i+n_1}$ to $P_{C^{(2)}}$.
Order ideals $I \subseteq P_C$ have an
accompanying decomposition
$I = I_1 \sqcup I_2$, making the distributive
lattice $J(P_C)$ a Cartesian product
$J(P_C) \cong  J(P_{C^{(1)}}) \times J(P_{C^{(2)}})$.

\begin{example} \rm
\label{ex:intro-example-5}
The permutation $w=3412756$ in $S_7$ from
Examples~\ref{ex:intro-example},
\ref{ex:intro-example-2}, \ref{ex:intro-example-3}, and~\ref{ex:intro-example-4} is decomposable as $w=3412 \oplus312$ in $S_4 \times S_3$,
corresponding to its disconnected heap poset
$P_C = P_{C^{(1)}} \sqcup P_{C^{(2)}}$.
Correspondingly, $J(P_C) \cong  J(P_{C^{(1)}}) \times J(P_{C^{(2)}})$, where
$J(P_{C^{(1)}})$ and $J(P_{C^{(2)}})$ are shown in Figure~\ref{fig:disconnected heap poset}.

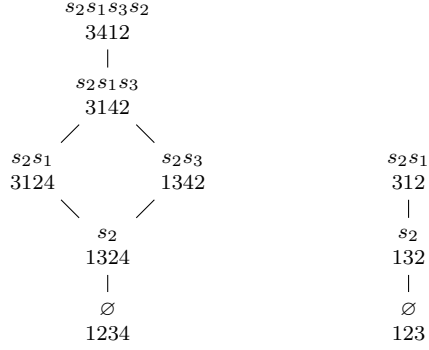
\begin{figure}[htbp]
\begingroup\Large\begin{tikzpicture}
  [scale=.2,auto=left]
  \node (e) at (0,0) {$\substack{\varnothing\\[.5ex]1234}$};
  \node (2) at (0,5) {$\substack{s_2\\[.5ex]1324}$};
  \node (21) at (-5,10) {$\substack{s_2s_1\\[.5ex]3124}$};
  \node (23) at (5,10) {$\substack{s_2s_3\\[.5ex]1342}$};
  \node (213) at (0,15) {$\substack{s_2s_1s_3\\[.5ex]3142}$};
  \node (2132) at (0,20) {$\substack{s_2s_1s_3s_2\\[.5ex]3412}$};
  \foreach \from/\to in {e/2,  2/21,2/23,21/213,23/213,
213/2132}
    \draw (\from) -- (\to);
  \node (e) at (20,0) {$\substack{\varnothing\\[.5ex]123}$};
  \node (6) at (20,5) {$\substack{s_2\\[.5ex]132}$};
  \node (65) at (20,10) {$\substack{s_2s_1\\[.5ex]312}$};
  \foreach \from/\to in {e/6,6/65}
    \draw (\from) -- (\to);
\end{tikzpicture}
\caption{The distributive lattices $J(P_{C^{(1)}})$ and $J(P_{C^{(2)}})$
corresponding to the components $P_{C^{(1)}}$ and $P_{C^{(2)}}$ of the heap poset $P_C$ for the
commuting class $C=C(\ii)$ with $\ii=(2,1,3,2,6,5)$. This $\ii$ is a reduced word for the decomposable permutation $w=3412756=3412 \oplus 312=$ in $S_4 \times S_3$. The lattice $J(P_C)$ from Figure~\ref{fig:distributive lattice of 3412756} has
$J(P_C) \cong 
J(P_{C^{(1)}}) \times J(P_{C^{(2)}})$.} \label{fig:disconnected heap poset}
\endgroup
\end{figure}   
\end{example}

The next proposition should come as no surprise.

\begin{proposition}
\label{prop:Fubini-for-decomposables}
Let $w=w^{(1)} \oplus w^{(2)}$ be a decomposable permutation, and $C$ a 
commuting class of reduced words for $w$, with $C^{(1)}$ and $C^{(2)}$ the corresponding commuting classes
of reduced words for $w^{(1)}$ and $w^{(2)}$, respectively.
Let $\prefix(C)$, $\prefix(C^{(1)})$, and $\prefix(C^{(2)})$
be their corresponding Condorcet domains of tiling type,
and $\voting$, $\voting^{(1)}$, and $\voting^{(2)}$ the uniform
vote tallies on each of these domains.

For $i=1,2$, let $u^{(i)}$ and $v^{(i)}$ be the permutations defined by Equations~\eqref{eqn:defining u as strict} and~\eqref{eqn:defining v as weak} in Theorem~\ref{thm:calcuating-maj-rule}, so that the uniform vote tally $\voting^{(i)}=\one_{\prefix(C^{(i)})}$ has majority rule $\majrule{\voting^{(i)}}$
equal to the intersection $<_{u^{(i)}} \cap <_{v^{(i)}}$.

Then the uniform vote tally $\voting=\one_{\prefix(C)}$ has $\majrule{\voting}$ equal to the intersection
$<_u \cap <_v$, where
\begin{align}
\label{eq:u-decomposition}    
u&=u^{(1)} \oplus u^{(2)} \text{ and}\\
\label{eq:v-decomposition}
v&=v^{(1)} \oplus v^{(2)}.
\end{align}
\end{proposition}
\begin{proof}
For ideals $I$ in $J(P_C)$, the unique decomposition $I = I_1 \sqcup I_2$ with $I_i$
an ideal of $J(P_{C^{(i)}})$  implies
\begin{equation}
    \label{eq:total-vote-tally-product}
    |\voting|=|J(P_C)|=
    |J(P_{C^{(1)}})| \cdot |J(P_{C^{(2)}})|=
    |\voting^{(1)}| \cdot |\voting^{(2)}|.
\end{equation}
Given an inversion $ab$ in $
\Inv(w)=P_C =P_{C^{(1)}} \sqcup P_{C^{(2)}},
$
assume, without loss of generality, that $ab$ lies in $\Inv(w^{(1)})=P_{C^{(1)}}$.
Then one
has $ab \in I$ if and only if $ab \in I_1$.
Consequently
$$
\Sigma_\voting(ab)
=|\{I_1 \in J(P_{C^{(1)}})|
\cdot |J(P_{C^{(2)}})|
=\Sigma_{\voting^{(1)}}(ab) \cdot |\voting^{(2)}|.
$$
Comparing this with~\eqref{eq:total-vote-tally-product}, one sees that the various inequality comparison conditions $\geq, >, =$ between
$\Sigma_\voting(ab)$ and $\frac{1}{2}|\voting|$
occur if and only one has the same conditions between
$\Sigma_{\voting^{(1)}}(ab)$ and $\frac{1}{2}|\voting^{(1)}|$.  Therefore, the
unique $u,v$ and $u^{(i)},v^{(i)}$ 
given by Theorem~\ref{thm:calcuating-maj-rule}
for $C$ and $C^{(i)}$, for $i=1,2$, are related
as in Equations~\eqref{eq:u-decomposition} and~\eqref{eq:v-decomposition}.
\end{proof}
 
\begin{example} \rm
    Continuing the running Example~\ref{ex:intro-example-5},
    for $w=3412756=3412 \oplus 312$ and $C$ as before, one has $u^{(1)}=1324$ and $v^{(1)}=3142$,
    while $u^{(2)}=132=v^{(2)}$. Therefore
$$\begin{array}{ccccccc}
u&=&u^{(1)} \oplus u^{(2)}&=&1324 \oplus 132&=&1324576,\\
v&=&v^{(1)} \oplus v^{(2)}&=&3142 \oplus 132&=&3142576.
\end{array}$$
\end{example}

Proposition~\ref{prop:Fubini-for-decomposables} reduces the calculation of $\majrule{\voting}$
for $\voting = \one_{\prefix(C)}$ for domains $\prefix(C)$ 
of tiling type to those whose permutation $w$ is indecomposable.
This is a special case of a more general {\it Fubini-type} result with a similar proof, asserting the following.  

Assume one is given
voting tallies $\voting^{(i)}: S_{n_i} \rightarrow \NN$ for $i=1,2$, and define their
{\it product tally} $\voting: S_{n_1+n_2} \rightarrow \NN$ via
$
\voting(w)=0
$
for $w$ in $S_n \setminus S_{n_1} \times S_{n_2}$,
and $\voting(w)=\voting^{(1)}(w^{(1)}) \cdot  \voting^{(2)}(w^{(2)})$
if $w=w^{(1)} \oplus w^{(2)}$ in $S_{n_1} \times S_{n_2}$.

\begin{proposition}
In the above context, the majority relation $\majrule{\voting}$ on $[n]$ is the disjoint union of the relations $\majrule{\voting^{(i)}}$ on $[n_i]$ for $i=1,2$:  one has
$a \majrule{\voting} b$ if and only if 
\begin{itemize}
    \item either $1 \leq a,b \leq n_1$ and $a \majrule{\voting^{(1)}} b$ on $[n_1]$,
    \item
or $n_1+1 \leq a,b \leq n_1+n_2$ and
$a-n_1 \majrule{\voting^{(2)}} b-n_1$
on $[n_2]$.
\end{itemize}
\end{proposition}

\section{A question on the higher Bruhat order $B(n,2)$}
\label{sec:remarks-and-questions}

We mention here some data and a question suggested by the previous results, related to the \emph{higher Bruhat order} $B(n,2)$. 
The higher Bruhat orders $B(n,k)$ first appeared in work of Manin and Schechtman~\cite{ManinSchechtman}, and Ziegler's study \cite{Ziegler} of them is particularly informative for our discussion here. We are interested in the case $k=2$, which enjoys some extra properties proven by Felsner and Weil in \cite{FelsnerWeil}. 

First define the set $A(n,2)$ as the set of all linear orderings $<$ of pairs $\{ \{a,b\} : 1 \le a < b \le n\}$ in which, for each triple $a < b < c$, the linear ordering either has
$$\{a,b\} < \{a,c\} < \{b,c\},$$
or it has
$$\{b,c\} < \{a,c\} < \{a,b\}.$$
In the latter case, we say that $abc$ is an \emph{inversion triple} for the ordering $<$. The set $A(n,2)$ is in bijection with the set of reduced words $\ii=(i_1,i_2,\ldots,i_{\binom{n}{2}})$ for the longest permutation $w_0=(n,n-1,\ldots,3,2,1)$ in $S_n$:  the pair $(a,b)$ appears at position $j$ in the linear order $<$ if 
$\Inv(s_{i_1}s_{i_2} \cdots s_{i_{j-1}} s_{i_j}) \setminus  \Inv(s_{i_1}s_{i_2} \cdots s_{i_{j-1}}) = \{(a,b)\}$.  Said differently, an element of $A(n,2)$ describes an ordered sequence of inversions transforming the identity into the longest permutation $w_0$. 

The poset $B(n,2)$ may be identified
as a set with the collection of all commuting equivalence classes $C(\ii)$ of reduced words for $w_0$ in $S_n$.  There are two ways to define its poset structure, beginning with the fact \cite[Corollary~2.3]{Ziegler} that one has equality $C=C'$ for the commuting equivalence classes $C$ and $C'$ if and only if their associated orderings $<$ and $<'$ on the pairs $(a,b)$
have the same set of inversion triples;  call this set of inversion triples $\Inv_3(C)$. In  \cite{FelsnerWeil}, Felsner and Weil answer a question of Ziegler by showing that these two poset structures
define the same partial order on this set $B(n,2)$:
\begin{itemize}
\item[(a)] Define $C < C'$ if and only if
$\Inv_3(C) \subset \Inv_3(C')$.
\item[(b)] Take the transitive closure $<$ of
this covering relation:  
$C \lessdot C'$ if there are reduced words $\ii \in C$ and $\ii' \in C'$ that differ by replacement of a substring $(j,j+1,j)$ in $\ii$ by the substring $(j+1,j,j+1)$ in $\ii'$, as in \eqref{eq:Yang-Baxter-relation},
or equivalently, when $\Inv_3(C) \subset \Inv_3(C')=\Inv_3(C) \sqcup \{abc\}$ for some
unique triple $abc$.
\end{itemize}

Since each element $C$ of $B(n,2)$ is an equivalence class of reduced words,
we can compute the majority relation $\majrule{\voting}$ for the uniform vote tally $\voting=\one_{\prefix(C)}$ 
on the Condorcet domain $\prefix(C)$ of tiling type. Mimicking a figure from Ziegler~\cite[Figure 1]{Ziegler}, we depict in Figure~\ref{figure:redrawn Ziegler} below the higher Bruhat order $B(5,2)$, in which each element $C$ is labeled by the majority relation on the Condorcet domain $\prefix(C)$, and by the set $\Inv_3(C)$  of its inversion triples\footnote{To save space, for $C$ at higher ranks, we indicate the {\it complement} of $\Inv_3(C)$ within the set of all $\binom{5}{3}=10$ triples.} defining $C$.  This 
prompts the following question.

\begin{question}
    \label{question:higher Bruhat edges}
What can one say about the majority relations on the uniform vote tallies on the Condorcet domains $\prefix(C)$ and $\prefix(C')$, when $C \lessdot C'$ is a cover
relation in the higher Bruhat order $B(n,2)$?
\end{question}

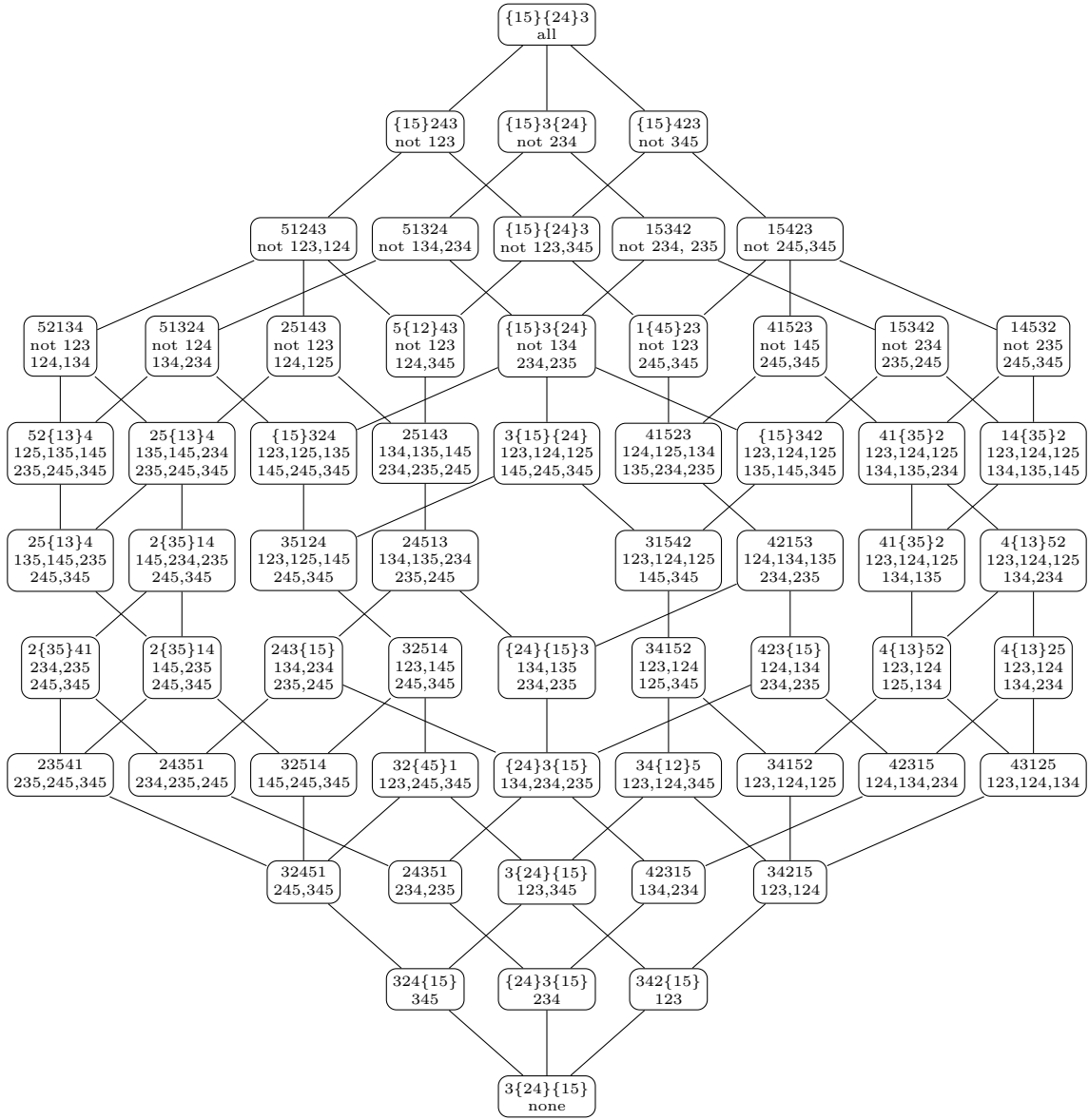
\begin{figure}[htbp]
\begingroup
\tikzset{every node/.style={draw, rounded corners, node font=\tiny}, every text node part/.style={align=center}}
$$\begin{tikzpicture}[yscale=1.5,xscale=1.7]
    \draw (0,0) node (nil) {$3\{24\}\{15\}$\\ none};
    \draw (-1,1) node (345) {$324\{15\}$\\345};
    \draw (0,1) node (234) {$\{24\}3\{15\}$\\234};
    \draw (1,1) node (123) {$342\{15\}$\\123};
    \draw (-2,2) node (245-345) {$32451$\\245,345};
    \draw (-1,2) node (234-235) {$24351$\\234,235};
    \draw (0,2) node (123-345) {$3\{24\}\{15\}$\\123,345};
    \draw (1,2) node (134-234) {$42315$\\134,234};
    \draw (2,2) node (123-124) {$34215$\\123,124};
    \draw (-4,3) node (235-245-345) {$23541$\\235,245,345};
    \draw (-3,3) node (234-235-245) {$24351$\\234,235,245};
    \draw (-2,3) node (145-245-345) {$32514$\\145,245,345};
    \draw (-1,3) node (123-245-345) {$32\{45\}1$\\123,245,345};
    \draw (0,3) node (134-234-235) {$\{24\}3\{15\}$\\134,234,235};
    \draw (1,3) node (123-124-345) {$34\{12\}5$\\123,124,345};
    \draw (2,3) node (123-124-125) {$34152$\\123,124,125};
    \draw (3,3) node (124-134-234) {$42315$\\124,134,234};
    \draw (4,3) node (123-124-134) {$43125$\\123,124,134};
    \draw (-4,4) node (234-235-245-345) {$2\{35\}41$\\234,235\\
    245,345};
    \draw (-3,4) node (145-235-245-345) {$2\{35\}14$\\145,235\\245,345};
    \draw (-2,4) node (134-234-235-245) {$243\{15\}$\\134,234\\235,245};
    \draw (-1,4) node (123-145-245-345) {$32514$\\123,145\\245,345};
    \draw (0,4) node (134-135-234-235) {$\{24\}\{15\}3$\\134,135\\234,235};
    \draw (1,4) node (123-124-125-345) {$34152$\\123,124\\125,345};
    \draw (2,4) node (124-134-234-235) {$423\{15\}$\\124,134\\234,235};
    \draw (3,4) node (123-124-125-134) {$4\{13\}52$\\123,124\\125,134};
    \draw (4,4) node (123-124-134-234) {$4\{13\}25$\\123,124\\134,234};
    \draw (-4,5) node (135-145-235-245-345) {$25\{13\}4$\\135,145,235\\245,345};
    \draw (-3,5) node (145-234-235-245-345) {$2\{35\}14$\\145,234,235\\245,345};
    \draw (-2,5) node (123-125-145-245-345) {$35124$\\123,125,145\\245,345};
    \draw (-1,5) node (134-135-234-235-245) {$24513$\\134,135,234\\235,245};
    \draw (1,5) node (123-124-125-145-345) {$31542$\\123,124,125\\145,345};
    \draw (2,5) node (124-134-135-234-235) {$42153$\\124,134,135\\234,235};
    \draw (3,5) node (123-124-125-134-135) {$41\{35\}2$\\123,124,125\\134,135};
    \draw (4,5) node (123-124-125-134-234) {$4\{13\}52$\\123,124,125\\134,234};
    \draw (-4,6) node (125-135-145-235-245-345) {$52\{13\}4$\\125,135,145\\235,245,345};
    \draw (-3,6) node (135-145-234-235-245-345) {$25\{13\}4$\\135,145,234\\235,245,345};
    \draw (-2,6) node (123-125-135-145-245-345) {$\{15\}324$\\123,125,135\\145,245,345};
    \draw (-1,6) node (134-135-145-234-235-245) {$25143$\\134,135,145\\234,235,245};
    \draw (0,6) node (123-124-125-145-245-345) {$3\{15\}\{24\}$\\123,124,125\\145,245,345};
    \draw (1,6) node (124-125-134-135-234-235) {$41523$\\124,125,134\\135,234,235};
    \draw (2,6) node (123-124-125-135-145-345) {$\{15\}342$\\123,124,125\\135,145,345};
    \draw (3,6) node (123-124-125-134-135-234) {$41\{35\}2$\\123,124,125\\134,135,234};
    \draw (4,6) node (123-124-125-134-135-145) {$14\{35\}2$\\123,124,125\\134,135,145};
    \draw (-4,7) node (125-135-145-234-235-245-345) {$52134$\\not 123\\124,134};
    \draw (-3,7) node (123-125-135-145-235-245-345) {$51324$\\not 124\\134,234};
    \draw (-2,7) node (134-135-145-234-235-245-345) {$25143$\\not 123\\124,125};
    \draw (-1,7) node (125-134-135-145-234-235-245) {$5\{12\}43$\\not 123\\124,345};
    \draw (0,7) node (123-124-125-135-145-245-345) {$\{15\}3\{24\}$\\not 134\\234,235};
    \draw (1,7) node (124-125-134-135-145-234-235) {$1\{45\}23$\\not 123\\245,345};
    \draw (2,7) node (123-124-125-134-135-234-235) {$41523$\\not 145\\245,345};
    \draw (3,7) node (123-124-125-134-135-145-345) {$15342$\\not 234\\235,245};
    \draw (4,7) node (123-124-125-134-135-145-234) {$14532$\\not 235\\245,345};
    \draw (-2,8) node (125-134-135-145-234-235-245-345) {$51243$\\not 123,124};
    \draw (-1,8) node (123-124-125-135-145-235-245-345) {$51324$\\not 134,234};
    \draw (0,8) node (124-125-134-135-145-234-235-245) {$\{15\}\{24\}3$\\not 123,345};
    \draw (1,8) node (123-124-125-134-135-145-245-345) {$15342$\\not 234, 235};
    \draw (2,8) node (123-124-125-134-135-145-234-235) {$15423$\\not 245,345};
    \draw (-1,9) node (124-125-134-135-145-234-235-245-345) {$\{15\}243$\\not 123};
    \draw (0,9) node (123-124-125-134-135-145-235-245-345) {$\{15\}3\{24\}$\\not 234};
    \draw (1,9) node (123-124-125-134-135-145-234-235-245) {$\{15\}423$\\not 345};
    \draw (0,10) node (123-124-125-134-135-145-234-235-245-345) {$\{15\}\{24\}3$\\all};
    \foreach \x in {345,234,123} {\draw (nil) -- (\x);}
    \foreach \x in {245-345,123-345} {\draw (345) -- (\x);}
    \foreach \x in {234-235,134-234} {\draw (234) -- (\x);}
    \foreach \x in {123-345,123-124} {\draw (123) -- (\x);}
    \foreach \x in {235-245-345,145-245-345,123-245-345} {\draw (245-345) -- (\x);}
    \foreach \x in {234-235-245,134-234-235} {\draw (234-235) -- (\x);}
    \foreach \x in {123-245-345,123-124-345} {\draw (123-345) -- (\x);}
    \foreach \x in {134-234-235,124-134-234} {\draw (134-234) -- (\x);}
    \foreach \x in {123-124-345,123-124-125,123-124-134} {\draw (123-124) -- (\x);}
    \foreach \x in {234-235-245-345,145-235-245-345} {\draw (235-245-345) -- (\x);}
    \foreach \x in {234-235-245-345,134-234-235-245} {\draw (234-235-245) -- (\x);}
    \foreach \x in {145-235-245-345,123-145-245-345} {\draw (145-245-345) -- (\x);}
    \draw (123-245-345) -- (123-145-245-345);
    \foreach \x in {134-234-235-245,134-135-234-235,124-134-234-235} {\draw (134-234-235) -- (\x);}
    \draw (123-124-345) -- (123-124-125-345);
    \foreach \x in {123-124-125-345,123-124-125-134} {\draw (123-124-125) -- (\x);}
    \foreach \x in {124-134-234-235,123-124-134-234} {\draw (124-134-234) -- (\x);}
    \foreach \x in {123-124-125-134,123-124-134-234} {\draw (123-124-134) -- (\x);}
    \draw (234-235-245-345) -- (145-234-235-245-345);
    \foreach \x in {135-145-235-245-345,145-234-235-245-345} {\draw (145-235-245-345) -- (\x);}
    \draw (134-234-235-245) -- (134-135-234-235-245);
    \draw (123-145-245-345) -- (123-125-145-245-345);
    \foreach \x in {134-135-234-235-245,124-134-135-234-235} {\draw (134-135-234-235) -- (\x);}
    \draw (123-124-125-345) -- (123-124-125-145-345);
    \draw (124-134-234-235) -- (124-134-135-234-235);
    \foreach \x in {123-124-125-134-135,123-124-125-134-234} {\draw (123-124-125-134) -- (\x);}
    \draw (123-124-134-234) -- (123-124-125-134-234);
    \foreach \x in {125-135-145-235-245-345,135-145-234-235-245-345} {\draw (135-145-235-245-345) -- (\x);}
    \draw (145-234-235-245-345) -- (135-145-234-235-245-345);
    \foreach \x in {123-125-135-145-245-345,123-124-125-145-245-345} {\draw (123-125-145-245-345) -- (\x);}
    \draw (134-135-234-235-245) -- (134-135-145-234-235-245);
    \foreach \x in {123-124-125-145-245-345,123-124-125-135-145-345} {\draw (123-124-125-145-345) -- (\x);}
    \draw (124-134-135-234-235) -- (124-125-134-135-234-235);
    \foreach \x in {123-124-125-134-135-234,123-124-125-134-135-145} {\draw (123-124-125-134-135) -- (\x);}
    \draw (123-124-125-134-234) -- (123-124-125-134-135-234);
    \foreach \x in {125-135-145-234-235-245-345,123-125-135-145-235-245-345} {\draw (125-135-145-235-245-345) -- (\x);}
    \foreach \x in {125-135-145-234-235-245-345,134-135-145-234-235-245-345} {\draw (135-145-234-235-245-345) -- (\x);}
    \foreach \x in {123-125-135-145-235-245-345,123-124-125-135-145-245-345} {\draw (123-125-135-145-245-345) -- (\x);}
    \foreach \x in {134-135-145-234-235-245-345,125-134-135-145-234-235-245} {\draw (134-135-145-234-235-245) -- (\x);}
    \draw (123-124-125-145-245-345) -- (123-124-125-135-145-245-345);
    \foreach \x in {124-125-134-135-145-234-235,123-124-125-134-135-234-235} {\draw (124-125-134-135-234-235) -- (\x);}
    \foreach \x in {123-124-125-135-145-245-345,123-124-125-134-135-145-345} {\draw (123-124-125-135-145-345) -- (\x);}
    \foreach \x in {123-124-125-134-135-234-235,123-124-125-134-135-145-234} {\draw (123-124-125-134-135-234) -- (\x);}
    \foreach \x in {123-124-125-134-135-145-345,123-124-125-134-135-145-234} {\draw (123-124-125-134-135-145) -- (\x);}
    \foreach \x in {125-135-145-234-235-245-345,134-135-145-234-235-245-345,125-134-135-145-234-235-245} {\draw (125-134-135-145-234-235-245-345) -- (\x);}
    \foreach \x in {123-125-135-145-235-245-345,123-124-125-135-145-245-345} {\draw (123-124-125-135-145-235-245-345) -- (\x);}
    \foreach \x in {125-134-135-145-234-235-245,124-125-134-135-145-234-235} {\draw (124-125-134-135-145-234-235-245) -- (\x);}
    \foreach \x in {123-124-125-135-145-245-345,123-124-125-134-135-145-345} {\draw (123-124-125-134-135-145-245-345) -- (\x) ;}
    \foreach \x in {124-125-134-135-145-234-235,123-124-125-134-135-234-235,123-124-125-134-135-145-234} {\draw (123-124-125-134-135-145-234-235) -- (\x);}
    \foreach \x in {125-134-135-145-234-235-245-345,124-125-134-135-145-234-235-245} {\draw (124-125-134-135-145-234-235-245-345) -- (\x);}
    \foreach \x in {123-124-125-135-145-235-245-345,123-124-125-134-135-145-245-345} {\draw (123-124-125-134-135-145-235-245-345) -- (\x);}
    \foreach \x in {124-125-134-135-145-234-235-245,123-124-125-134-135-145-234-235} {\draw (123-124-125-134-135-145-234-235-245) -- (\x);}
    \foreach \x in {124-125-134-135-145-234-235-245-345,123-124-125-134-135-145-235-245-345,123-124-125-134-135-145-234-235-245} {\draw (123-124-125-134-135-145-234-235-245-345) -- (\x);}
\end{tikzpicture}$$
\endgroup
\caption{Commutation classes of reduced words for $54321 \in S_5$, labeled by the majority relation on each class, together with the inversion triples of that class.}\label{figure:redrawn Ziegler}
\end{figure}

For example, if $C \lessdot C'$ in which
$\Inv_3(C')=\Inv_3(C) \sqcup \{abc\}$ with $a<b<c$, do the majority relations for $\prefix(C)$ and $\prefix(C')$ differ only in the relative orderings of the values within the range $a,a+1,\ldots,c-1,c$? An interesting example here is the class $C$ with 
$\Inv_3(C)=\{123, 124\}$, represented by the reduced word 
$$
(2, 3, 1, 2, 1, 3, 4, 3, 2, 1),
$$
for which $\voting=\one_{\prefix(C)}$ has majority relation
$
34215.
$
Adding one inversion triple 
$abc=134$ to give $\Inv_3(C')=\{123, 124, 134\}$ gives the class $C'$ represented by 
$$
(2, 3, {\mathbf{2}}, {\mathbf{1}}, {\mathbf{2}}, 3, 4, 3, 2, 1)
$$ 
with boldface indicating the move 
\eqref{eq:Yang-Baxter-relation} adding the inversion triple. Then $\voting=\one_{\prefix(C')}$ has majority relation
$
43125,
$
in which we see that the value $2$ has changed position between the two majority relations, not just the positions of the values $\{1,3,4\}$ from the inversion triple. The position of $5$, on the other hand, is unchanged.

\section*{Acknowledgments}
The authors thank Patrek K\'arason Ragnarsson
for supplying 
invaluable code to compute the majority rule $\majrule{\voting}$ for the uniform vote tally $\voting = \one_{\prefix(C)}$  
on a Condorcet domain 
of tiling type, given any reduced word in the commutation class $C$.
The first author was supported by NSF grant DMS-2450430, and the second author was supported by NSF grant DMS-2054436.

\bibliographystyle{abbrv}
\bibliography{references}

\begin{thebibliography}{10}

\bibitem{Abello}
J.~Abello.
\newblock The weak {B}ruhat order of {$S_\Sigma$} consistent sets, and
  {C}atalan numbers.
\newblock {\em SIAM J. Discrete Math.}, 4(1):1--16, 1991.

\bibitem{AngelHolroydRomikVirag}
O.~Angel, A.~E. Holroyd, D.~Romik, and B.~Vir\'ag.
\newblock Random sorting networks.
\newblock {\em Adv. Math.}, 215(2):839--868, 2007.

\bibitem{BanaianChepuriGunawanPan}
E.~Banaian, S.~Chepuri, E.~Gunawan, and J.~Pan.
\newblock $c$-birkhoff polytopes, 2025.

\bibitem{BjornerBrenti}
A.~Bj\"orner and F.~Brenti.
\newblock {\em Combinatorics of {C}oxeter groups}, volume 231 of {\em Graduate
  Texts in Mathematics}.
\newblock Springer, New York, 2005.

\bibitem{CartierFoata}
P.~Cartier and D.~Foata.
\newblock {\em Probl\`emes combinatoires de commutation et r\'earrangements},
  volume No. 85 of {\em Lecture Notes in Mathematics}.
\newblock Springer-Verlag, Berlin-New York, 1969.

\bibitem{Chameni-Nembua}
C.~Chameni-Nembua.
\newblock R\`egle majoritaire et distributivit\'e{} dans le permuto\`edre.
\newblock {\em Math. Inform. Sci. Humaines}, 108:5--22, 1989.

\bibitem{DanilovKarzanovKoshevoy}
V.~I. Danilov, A.~V. Karzanov, and G.~Koshevoy.
\newblock Condorcet domains of tiling type.
\newblock {\em Discrete Appl. Math.}, 160(7-8):933--940, 2012.

\bibitem{DanilovKarzanovKoshevoy2021}
V.~I. Danilov, A.~V. Karzanov, and G.~A. Koshevoy.
\newblock Majority rule on rhombus tilings and {C}ondorcet super-domains.
\newblock {\em Discrete Appl. Math.}, 292:85--96, 2021.

\bibitem{Elnitsky}
S.~Elnitsky.
\newblock Rhombic tilings of polygons and classes of reduced words in {C}oxeter
  groups.
\newblock {\em J. Combin. Theory Ser. A}, 77(2):193--221, 1997.

\bibitem{EscobarPechenikTennerYong}
L.~Escobar, O.~Pechenik, B.~E. Tenner, and A.~Yong.
\newblock Rhombic tilings and {B}ott-{S}amelson varieties.
\newblock {\em Proc. Amer. Math. Soc.}, 146(5):1921--1935, 2018.

\bibitem{FanStembridge}
C.~K. Fan and J.~R. Stembridge.
\newblock Nilpotent orbits and commutative elements.
\newblock {\em J. Algebra}, 196(2):490--498, 1997.

\bibitem{FelsnerWeil}
S.~Felsner and H.~Weil.
\newblock A theorem on higher {B}ruhat orders.
\newblock {\em Discrete Comput. Geom.}, 23(1):121--127, 2000.

\bibitem{Fishburn1997}
P.~Fishburn.
\newblock Acyclic sets of linear orders.
\newblock {\em Soc. Choice Welf.}, 14(1):113--124, 1997.

\bibitem{Fishburn2002}
P.~C. Fishburn.
\newblock Acyclic sets of linear orders: a progress report.
\newblock {\em Soc. Choice Welf.}, 19(2):431--447, 2002.

\bibitem{GalambosReiner}
A.~Galambos and V.~Reiner.
\newblock Acyclic sets of linear orders via the {B}ruhat orders.
\newblock {\em Soc. Choice Welf.}, 30(2):245--264, 2008.

\bibitem{GoodmanPollack}
J.~E. Goodman.
\newblock Pseudoline arrangements.
\newblock In {\em Handbook of discrete and computational geometry}, CRC Press
  Ser. Discrete Math. Appl., pages 83--109. CRC, Boca Raton, FL, 1997.

\bibitem{GreenLosonczy}
R.~M. Green and J.~Losonczy.
\newblock Freely braided elements of {C}oxeter groups.
\newblock {\em Ann. Comb.}, 6(3-4):337--348, 2002.

\bibitem{Humphreys}
J.~E. Humphreys.
\newblock {\em Reflection groups and {C}oxeter groups}, volume~29 of {\em
  Cambridge Studies in Advanced Mathematics}.
\newblock Cambridge University Press, Cambridge, 1990.

\bibitem{Knuth}
D.~E. Knuth.
\newblock {\em Axioms and hulls}, volume 606 of {\em Lecture Notes in Computer
  Science}.
\newblock Springer-Verlag, Berlin, 1992.

\bibitem{LabbeLange}
J.-P. Labb\'e and C.~E. M.~C. Lange.
\newblock Cambrian acyclic domains: counting {$c$}-singletons.
\newblock {\em Order}, 37(3):571--603, 2020.

\bibitem{ManinSchechtman}
Y.~I. Manin and V.~V. Schechtman.
\newblock Arrangements of hyperplanes, higher braid groups and higher {B}ruhat
  orders.
\newblock In {\em Algebraic number theory}, volume~17 of {\em Adv. Stud. Pure
  Math.}, pages 289--308. Academic Press, Boston, MA, 1989.

\bibitem{Matsummoto}
H.~Matsumoto.
\newblock G\'en\'erateurs et relations des groupes de {W}eyl g\'en\'eralis\'es.
\newblock {\em C. R. Acad. Sci. Paris}, 258:3419--3422, 1964.

\bibitem{Monjardet}
B.~Monjardet.
\newblock Acyclic domains of linear orders: {A} survey.
\newblock In {\em The mathematics of preference, choice and order}, Stud.
  Choice Welf., pages 139--160. Springer, Berlin, 2009.

\bibitem{PuppeSlinko}
C.~Puppe and A.~Slinko.
\newblock Condorcet domains, median graphs and the single-crossing property.
\newblock {\em Econom. Theory}, 67(1):285--318, 2019.

\bibitem{Reading}
N.~Reading.
\newblock Noncrossing partitions, clusters and the {C}oxeter plane.
\newblock {\em S\'em. Lothar. Combin.}, 63:Art. B63b, 32, 2010.

\bibitem{ReinerRoichman}
V.~Reiner and Y.~Roichman.
\newblock Diameter of graphs of reduced words and galleries.
\newblock {\em Trans. Amer. Math. Soc.}, 365(5):2779--2802, 2013.

\bibitem{SchillingThieryWhiteWilliams}
A.~Schilling, N.~M. Thi\'ery, G.~White, and N.~Williams.
\newblock Braid moves in commutation classes of the symmetric group.
\newblock {\em European J. Combin.}, 62:15--34, 2017.

\bibitem{Stanley-EC1}
R.~P. Stanley.
\newblock {\em Enumerative combinatorics. {V}olume 1}, volume~49 of {\em
  Cambridge Studies in Advanced Mathematics}.
\newblock Cambridge University Press, Cambridge, second edition, 2012.

\bibitem{Stembridge}
J.~R. Stembridge.
\newblock On the fully commutative elements of {C}oxeter groups.
\newblock {\em J. Algebraic Combin.}, 5(4):353--385, 1996.

\bibitem{Tenner-Repetition}
B.~E. Tenner.
\newblock Repetition in reduced decompositions.
\newblock {\em Adv.~Appl.~Math.}, 49:1--14, 2012.

\bibitem{Tenner-OneElement}
B.~E. Tenner.
\newblock One-element commutation classes.
\newblock {\em J. Combin.}, 15(3):401--408, 2024.

\bibitem{Tits}
J.~Tits.
\newblock Le probl\`eme des mots dans les groupes de {C}oxeter.
\newblock In {\em Symposia {M}athematica ({INDAM}, {R}ome, 1967/68), {V}ol. 1},
  pages 175--185. Academic Press, London-New York, 1969.

\bibitem{Viennot}
G.~X. Viennot.
\newblock Heaps of pieces. {I}. {B}asic definitions and combinatorial lemmas.
\newblock In {\em Graph theory and its applications: {E}ast and {W}est
  ({J}inan, 1986)}, volume 576 of {\em Ann. New York Acad. Sci.}, pages
  542--570. New York Acad. Sci., New York, 1989.

\bibitem{Ziegler}
G.~M. Ziegler.
\newblock Higher {B}ruhat orders and cyclic hyperplane arrangements.
\newblock {\em Topology}, 32(2):259--279, 1993.

\end{thebibliography}

\end{document}